\documentclass[11pt,reqno,a4paper]{amsart}
\usepackage{amsmath,amssymb,amsthm}
\usepackage[margin=1in]{geometry}
\usepackage{color}
\usepackage[colorlinks=true]{hyperref}

\newtheorem{thm}{Theorem}[section]
\newtheorem{cor}[thm]{Corollary}
\newtheorem{lem}[thm]{Lemma}
\newtheorem{prop}[thm]{Proposition}

\theoremstyle{definition}

\newtheorem{rem}[thm]{Remark}
\newtheorem*{ack}{Acknowledgement}
\numberwithin{equation}{section}
\allowdisplaybreaks[1]

\usepackage{etoolbox}
\AtBeginEnvironment{proof}{\vspace{-11pt}}
\AtEndEnvironment{proof}{\vspace{-4pt}}

\def\bR{\mathbb{R}}
\def\bC{\mathbb{C}}

\def\bN{\mathbb{N}}
\def\rd{\mathrm{d}}

\DeclareMathOperator\UU{U}
\DeclareMathOperator\SO{SO}
\DeclareMathOperator\SL{SL}
\DeclareMathOperator\vol{vol}
\DeclareMathOperator\Vol{Vol}
\DeclareMathOperator\Ree{Re}
\DeclareMathOperator\Span{Span}
\DeclareMathOperator\Ker{Ker}
\DeclareMathOperator\Imm{Im}

\title{Remarks on the self-shrinking Clifford torus}
\author{Christopher G. Evans}
\author{Jason D. Lotay}
\author{Felix Schulze}

\address{Department of Mathematics, University College London, Gower Street, London, WC1E 6BT, United Kingdom}

\email{christopher.evans.13@ucl.ac.uk, j.lotay@ucl.ac.uk,  f.schulze@ucl.ac.uk}

\date{\today}

\begin{document}

\begin{abstract}
On the one hand, we prove that the Clifford torus in $\bC^2$ is unstable for Lagrangian mean curvature flow under arbitrarily small Hamiltonian perturbations, even though it is Hamiltonian $F$-stable and locally area minimising under Hamiltonian variations.  On the other hand, we show that the Clifford torus is rigid: it is locally unique as a self-shrinker for mean curvature flow, despite having infinitesimal deformations which do not arise from rigid motions.  The proofs rely on analysing higher order phenomena: specifically, showing that the Clifford torus is not a local entropy minimiser even under Hamiltonian variations, and demonstrating that infinitesimal deformations which do not generate rigid motions are genuinely obstructed.
\end{abstract}
\maketitle

\section{Introduction}

The Clifford torus contained in the 3-sphere in $\bC^2$ is an important example of a self-shrinker in mean curvature flow.  Moreover, the Clifford torus is Lagrangian in $\bC^2$ and has  particular significance in Lagrangian mean curvature flow: it is the simplest known example of a compact Lagrangian self-shrinker in $\bC^2$, as there are no self-shrinking Lagrangian spheres in $\bC^2$ \cite{Smoczyk.Hab} (even allowing for branched immersed spheres \cite{ChenMa}).

In this article we study two related issues: stability of the Clifford torus under (Lagrangian) mean curvature flow, and rigidity of the Clifford torus as a (Lagrangian) self-shrinker.   Knowing stability would imply rigidity, but the converse is not necessarily true, as will be the case here.  
 Both issues are clearly crucial for understanding problems such as singularity formation and genericity of singularity models.

\vspace{-4pt}
\subsection{Hamiltonian instability}  Given that the Clifford torus is the simplest example of a Lagrangian self-shrinker, Neves  \cite[cf.~Question 7.4]{NevesSurvey} asked under what conditions on a Lagrangian torus would the rescaled Lagrangian mean curvature flow converge to the Clifford torus (up to translations and unitary transformations).  It is known, for example by work in \cite{LeeLue, LiZhang}, that the Clifford torus is unstable under Lagrangian mean curvature flow, even at the linear level (i.e.~it is Lagrangian $F$-unstable).  However, the variations used there to prove instability are not Hamiltonian: they are the variations where one shrinks the size of one circle generator in the Clifford torus relative to the other.  

In fact, the Clifford torus is locally area-minimising under Hamiltonian variations, see \cite{Oh} as well as Theorem \ref{thm:L.localmin.vol},  and Oh conjectured that the Clifford torus is globally area-minimising in its Hamiltonian isotopy class.  Therefore, one would expect that it would be stable for Lagrangian mean curvature flow under sufficiently small Hamiltonian perturbations.   This expectation is reinforced by the fact that the Clifford torus is Hamiltonian $F$-stable, see \cite{LeeLue,LiZhang} as well as Lemma \ref{lem:L.HamFstable}.

However,  in spite of this, we show the following surprising phenomenon.

\begin{thm}\label{thm:Cliff.unstable}
The Clifford torus is unstable for Lagrangian mean curvature flow under arbitrarily $C^k$-small Hamiltonian perturbations for any $k\geq 0$.
\end{thm}

The precise statement can be found in Theorem \ref{thm:flow.instability}.    Our construction is explicit, and shows the result holds even in the $\UU(1)$-equivariant setting: in this context, the statement is that the circle is unstable for the $\UU(1)$-equivariant Lagrangian mean curvature flow under arbitrarily small Hamiltonian deformations.  By looking at the $\UU(1)$-equivariant flow, it was shown in \cite{Grohetal, NevesMonotone} that the Clifford torus was unstable under large Hamiltonian perturbations, and in \cite{NevesFTS} that the Clifford torus is unstable under arbitrarily $C^0$-small Hamiltonian perturbations (but this argument would never give $C^1$-small perturbations due to the nature of the construction).   Theorem \ref{thm:Cliff.unstable} therefore improves these results in this particular setting.

We expect that for the unstable perturbations the Lagrangian mean curvature flow develops a first finite-time singularity, which is Type II, whose Type I blow-up is a transverse pair of special Lagrangian planes (with the same Lagrangian angle).  

\vspace{-4pt}
\subsection{Hamiltonian stability} We observe by the work in \cite{CasLerMiq} we have the following stability result (cf.~Theorem \ref{thm:Hamiso.CLM}).

\begin{thm}\label{thm:Cliff.stable}
A compact embedded Lagrangian $L_0$ in the 3-sphere in $\bC^2$ is Hamiltonian isotopic to the Clifford torus if and only if Lagrangian mean curvature flow starting at $L_0$, after rescaling, converges to the Clifford torus (up to unitary transformation).
\end{thm}

This is really a manifestation of the fact that a simple closed curve $\gamma_0$ in the standard 2-sphere is Hamiltonian isotopic to an equator if and only if curve shortening flow starting at $\gamma_0$ exists for all time and converges to an equator.  

\vspace{-4pt}
\subsection{Rigidity}
Given that the Clifford torus is Hamiltonian $F$-stable, the instability result Theorem \ref{thm:Cliff.unstable} is only possible because there are infinitesimal (Hamiltonian) deformations of the Clifford torus as a self-shrinker which do not come from translations, dilations or rotations.  This means that, a priori, it is not clear whether the Clifford torus is locally isolated in the space of self-shrinkers or not, and standard methods cannot determine the rigidity or otherwise of the Clifford torus.  

However, we use a novel approach to rigidity to show the Clifford torus indeed has this property.

\begin{thm}\label{thm:Cliff.rigid}
The Clifford torus is locally unique as a self-shrinker for mean curvature flow.
\end{thm}

The precise statement can be found in Theorem \ref{thm:local.uniq}.   The rigidity of the Clifford torus is perhaps expected given that it is a simple and natural example of a self-shrinker, but what is surprising is that the proof of this result does not, and cannot, proceed as one might expect.  Indeed, we hope that the novel method we employ to prove Theorem \ref{thm:Cliff.rigid} will be useful in other contexts.

\vspace{-4pt}
\subsection{Entropy}
The proof of Theorem \ref{thm:Cliff.unstable} relies on explicitly showing that the Clifford torus is not a local minimizer for the entropy \cite{ColdMin} under Hamiltonian variations.  Due to Hamiltonian $F$-stability, we know that this is not an issue that can be analysed at the ``linear level''.  

More precisely, any Hamiltonian variation of order $O(s)$ for which the entropy could go down must  
have an entropy value which agrees with that of the Clifford torus up to and including order $O(s^2)$.  Therefore, one needs to look at ``higher order'' terms.  It transpires that the first order at which the entropy could go down is $O(s^6)$, showing the delicate nature of the problem.  

There is an additional issue that the entropy is defined as a supremum over all space-time points, so it is not practical to compute directly.  We overcome this through an argument which allows us to restrict attention to the $F$-functional, which can be computed. 

Theorem \ref{thm:Cliff.unstable} follows from monotonicity of the entropy under (rescaled) mean curvature flow.

\vspace{-4pt}
\subsection{Obstructions}
Since we have infinitesimal deformations of the Clifford torus as a self-shrinker which do not come from rotations, we have to demonstrate that these infinitesimal deformations do not extend to genuine deformations.  Therefore, again we have to go beyond the ``linear level'' in the analysis, and take a new approach to the study of local uniqueness.

More concretely, if the deformation is of order $O(s)$ we have to explicitly demonstrate that there are obstructions to extending it to a solution of the self-shrinker equation at order $O(s^k)$ for some $k\geq 2$.  Here, we  view the problem of solving a nonlinear equation in terms of its linearisation and an iterative fixed point/contraction mapping argument, as one uses in the Implicit Function Theorem. It turns out that obstructions do not appear at the first step (i.e.~$k=2$) but rather at $O(s^3)$.  Again, this demonstrates the somewhat subtle nature of the problem.

\vspace{-4pt}
\subsection{Summary}
We now briefly summarise the contents of the article.  

In $\S$\ref{sec:prelim} we introduce the notation we shall use throughout for studying the Clifford torus and its deformations, as well as recalling basic facts from Lagrangian geometry and the definition of the entropy and $F$-functional.   We describe some of the eigenspaces for low eigenvalues of the Laplacian on functions and on the normal bundle of the Clifford torus, since this plays a role in the second variation of $F$, which we derive.

In $\S$\ref{sec:orbits}, we study the orbit of the Clifford torus under various relevant group actions and show that generators of these actions correspond to elements of the eigenspaces we described in $\S$\ref{sec:prelim}.

In $\S$\ref{sec:instability}, after identifying the variations giving Hamiltonian $F$-stability, for one of these sufficiently small variations we compute the $F$-functional and the entropy.  By showing it goes down, we prove Theorem \ref{thm:Cliff.unstable}.  We also prove Theorem \ref{thm:Cliff.stable} in this section using work in \cite{CasLerMiq} and some elementary observations, and additionally give an alternative proof of a weaker form of the stability result Theorem \ref{thm:Cliff.stable} which may have applications in other contexts.

Finally, in $\S$\ref{sec:rigidity}, we set-up the deformation problem for self-shrinkers in terms of zeros of a smooth map.  After gauge-fixing for the action of rotations, we obtain a nonlinear elliptic operator acting on normal vector fields whose zeros characterise nearby self-shrinkers.  We identify the (self-adjoint) linearisation of this operator, its kernel, and show that the nonlinear operator determines a non-trivial cubic map from the kernel to itself.  From this, we prove Theorem \ref{thm:Cliff.rigid}.

\begin{ack}
This research was supported by Leverhulme Trust Research Project Grant RPG-2016-174.  The last author would like to thank Jonathan Zhu for encouraging the authors to provide further details in a part of the argument leading to Theorem \ref{thm:Cliff.rigid}.
\end{ack}

\section{Preliminaries}\label{sec:prelim}

We define the Clifford torus in $\mathcal{S}^3(2)\subseteq\bC^2=\bR^4$ by
\begin{equation*}
L=\{\sqrt{2}(e^{i\theta_1},e^{i\theta_2})\in\bC^2\,:\,\theta_1,\theta_2\in\bR\}.
\end{equation*}
We then have that $L$ is Lagrangian and minimal in $\mathcal{S}^3(2)$.  Therefore, if $X$ denotes the position vector on $L$, we have that $X=X^{\perp}$, where ${}^{\perp}$ denotes projection onto the normal bundle $NL$.  Moreover, the mean curvature vector $H$ of $L$ in $\bC^2$ satisfies 
\begin{equation}\label{eq:shrinker}
H=-\frac{X^{\perp}}{2},
\end{equation}
i.e.~$L$ is a self-shrinker so that $L_t=\sqrt{1-t}L$ is a solution to (Lagrangian) mean curvature flow with $L_0=L$.   These facts are easy to check so we do it here, as the computations will be useful later.

Throughout we will use complex coordinates $z_1,z_2$ on $\bC^2$ and corresponding real coordinates $(x_1,y_1,x_2,y_2)$ on $\bR^4$ so that $z_1=x_1+iy_1$ and $z_2=x_2+iy_2$.

\vspace{-4pt}
\subsection{Basics}
We let $\mathcal{S}^1$ be the unit circle in $\bC$.  We also let $J$ and $\omega$ be the standard complex and symplectic structures on $\bC^2$.  We define the embedding $X:\mathcal{S}^1\times\mathcal{S}^1\to L\subseteq\bC^2$ by
\begin{equation*}
X(e^{i\theta_1},e^{i\theta_2})=\sqrt{2}(e^{i\theta_1},e^{i\theta_2}).
\end{equation*}
We therefore have two tangent vector fields on $L$ given by 
\begin{equation*}
X_1=X_*\left(\frac{\partial}{\partial\theta_1}\right)=\sqrt{2}(ie^{i\theta_1},0)\quad\text{and}\quad X_2=X_*\left(\frac{\partial}{\partial\theta_2}\right)=\sqrt{2}(0,ie^{i\theta_2}).
\end{equation*}
It is immediate that
\begin{equation}\label{eq:L.metric.0}
|X_1|^2=|X_2|^2=2\quad\text{and}\quad \langle X_1,X_2\rangle=0,
\end{equation}
thus the induced metric on $L$ is
\begin{equation}\label{eq:L.metric}
2\rd\theta_1^2+2\rd\theta_2^2,
\end{equation}
which is flat.

It is clear by inspection that 
\begin{equation*}
JX_1=-\sqrt{2}(e^{i\theta_1},0)\quad\text{and}\quad JX_2=-\sqrt{2}(0,e^{i\theta_2})
\end{equation*}
are orthogonal unit normal vector fields on $L$, and hence $L$ is Lagrangian.  Moreover,
\begin{equation}\label{eq:L.X}
X=X^{\perp}=-JX_1-JX_2.
\end{equation}
We see that 
\begin{equation*}
E_1=(ie^{i\theta_1},0)=\frac{1}{\sqrt{2}}X_1\quad\text{and}\quad
E_2=(0,ie^{i\theta_2})=\frac{1}{\sqrt{2}}X_2
\end{equation*}
are orthonormal tangent vector fields on $L$, so we can compute the mean curvature vector field
\begin{equation}\label{eq:L.H}
\begin{split}
H&=\nabla_{E_1}E_1+\nabla_{E_2}E_2=\frac{1}{2}\nabla_{X_1}X_1+\frac{1}{2}\nabla_{X_2}X_2\\
&=-\frac{\sqrt{2}}{2}(e^{i\theta_1},0)-\frac{\sqrt{2}}{2}(0,e^{i\theta_2})=\frac{1}{2}(JX_1+JX_2)=-\frac{X^{\perp}}{2},
\end{split}
\end{equation}
as claimed in \eqref{eq:shrinker}.  

Observe that if $\Omega=\rd z_1\wedge\rd z_2$ is the standard holomorphic volume form on $\bC^2$ then, since
\begin{equation*}
X^*\rd z_1=\sqrt{2}ie^{i\theta_1}\rd\theta_1\quad\text{and}\quad X^*\rd z_2=\sqrt{2}ie^{i\theta_2}\rd\theta_2,
\end{equation*}
we have that
\begin{equation*}
X^*\Omega=-2e^{i(\theta_1+\theta_2)}\rd\theta_1\wedge\rd\theta_2.
\end{equation*}
Since the volume form on $L$ is 
\begin{equation}\label{eq:vol}
\vol_L=2\rd\theta_1\wedge\rd\theta_2,
\end{equation}
we see that the Lagrangian angle $\theta$ of $L$ satisfies
\begin{equation*}
e^{i\theta}=-e^{i(\theta_1+\theta_2)}=e^{i(\theta_1+\theta_2+\pi)}.
\end{equation*} 
We therefore verify that
\begin{equation*}
H=J\nabla\theta=J\nabla(\theta_1+\theta_2+\pi)=\frac{1}{2}(JX_1+JX_2).
\end{equation*}
(Notice that the factor of $\frac{1}{2}$ arises from \eqref{eq:L.metric.0}.)

The Clifford torus is also an example of a (positive) monotone Lagrangian, i.e.~if $\lambda$ is the 1-form dual to $JX^{\perp}$ then 
\begin{equation*}
[\rd\theta]=\frac{c}{2}[\lambda]\in H^1(L)
\end{equation*}
for some constant $c>0$.  We know in fact that $c=1$ for any Lagrangian self-shrinker satisfying \eqref{eq:shrinker}.  The monotone property is preserved under Hamiltonian isotopy, and any monotone Lagrangian can be rescaled so that $c=1$.

\vspace{-4pt}
\subsection{Laplacians}  We shall see that to understand the stability properties of the Clifford torus as a self-shrinker, we will need to analyse the Laplacian on normal vector fields, and particularly the $1$-eigenspace of the Laplacian.

 Notice that \eqref{eq:L.metric} implies that the induced Laplacian on functions on $L$ is just
\begin{equation}\label{eq:L.Delta}
 \Delta_L=\frac{1}{2}\Delta_0=-\frac{1}{2}\left(\frac{\partial^2}{\partial\theta_1^2}+\frac{\partial^2}{\partial\theta_2^2}\right),
\end{equation}
where $\Delta_0$ is just the standard Laplacian on $\mathcal{S}^1\times\mathcal{S}^1$.  In particular, we have the following elementary facts.

\begin{lem}\label{lem:f.1.espace}
For $\Delta_L$ given in \eqref{eq:L.Delta} we have that the set of eigenvalues is
\begin{equation*}
\left\{\textstyle\frac{1}{2}n\,:\,n\in\bN\right\}.
\end{equation*}
Moreover, we have:
\begin{gather*}
\Span\{1\}=\{f\,:\,\Delta_Lf=0\};\\
\Span\{\cos\theta_1,\sin\theta_1,\cos\theta_2,\sin\theta_2\}=\{f\,:\,\Delta_Lf=\textstyle\frac{1}{2} f\};\\
\Span\{\cos(\theta_1+\theta_2),\sin(\theta_1+\theta_2),\cos(\theta_1-\theta_2),\sin(\theta_1-\theta_2)\}=\{f:\Delta_L f=f\}.
\end{gather*}
\end{lem}

\begin{proof}
This is immediate from \eqref{eq:L.Delta}, the observation that
\begin{equation*}
\Delta_0\cos(\theta_1\pm\theta_2)=2\cos(\theta_1\pm\theta_2)\quad\text{and}
\quad
\Delta_0\sin(\theta_1\pm\theta_2)=2\sin(\theta_1\pm\theta_2),
\end{equation*}
and the fact that $\cos\theta_i$, $\sin\theta_i$ are $1$-eigenfunctions of $\Delta_0$ for $i=1,2$.
\end{proof}

Moreover, as $L$ is Lagrangian and flat, and the normal bundle and tangent bundle are isometric, we have that the normal bundle of $L$ is flat.  Thus the Laplacian on the normal bundle is given by the rough Laplacian
\begin{equation}\label{eq:L.nDelta.rough}
\Delta_L^{\perp}=-\nabla_{E_1}^{\perp}\nabla_{E_1}^{\perp}-\nabla_{E_2}^{\perp}\nabla_{E_2}^{\perp}.
\end{equation}

It is easy to see that
\begin{equation}\label{eq:L.nabla}
\nabla_{X_1}X_1=JX_1,\quad \nabla_{X_2}X_2=JX_2,\quad \nabla_{X_1}X_2=\nabla_{X_2}X_1=0.
\end{equation} 
Therefore, as the complex structure $J$ is parallel (or just by inspection),
\begin{gather*}
\nabla_{X_1}(JX_1)=J\nabla_{X_1}X_1=-X_1, \qquad \nabla_{X_2}(JX_2)=J\nabla_{X_2}X_2=-X_2,\\
\nabla_{X_1}(JX_2)=J\nabla_{X_1}X_2=0,\qquad \nabla_{X_2}(JX_1)=J\nabla_{X_2}X_1=0.
\end{gather*}
We therefore see that 
\begin{equation}\label{eq:n.diff.zero}
\nabla_{X_i}^{\perp}(JX_j)=0
\end{equation}
for all $i,j$.  Hence, $JX_1$ and $JX_2$ are harmonic normal vector fields:
\begin{equation*}
\Delta_L^{\perp}JX_i=0.
\end{equation*}
Notice this implies that
\begin{equation}\label{eq:H.harmonic}
\Delta_L^{\perp}H=0=\Delta_L^{\perp}X^{\perp}.
\end{equation}

Since $X_1$ and $X_2$ span the normal vector fields on $L$, we can write any normal vector field $V$ on $L$ uniquely as
\begin{equation}\label{eq:L.normal}
V=f_1JX_1+f_2JX_2.
\end{equation}
From \eqref{eq:n.diff.zero} we have that
\begin{equation}\label{eq:L.nDelta}
\Delta_L^{\perp}V=(\Delta_Lf_1)JX_1+(\Delta_Lf_2)JX_2,
\end{equation}
where $\Delta_L$ is given in \eqref{eq:L.Delta}.  This yields the following.

\begin{lem}\label{lem:V.1.espace}
For $\Delta_L^{\perp}$ given in \eqref{eq:L.nDelta.rough} we have that the set of eigenvalues is
\begin{equation*}
\left\{\textstyle\frac{1}{2}n\,:\,n\in\bN\right\}.
\end{equation*}
Moreover, we have:
\begin{gather*}
\Span\{JX_1+JX_2,JX_1-JX_2\}=\{V\,:\,\Delta_L^{\perp}V=0\};\\
\Span\{(\cos\theta_j)JX_k,(\sin\theta_j)JX_k\,:\,j,k=1,2\}=\{V\,:\,\Delta_L^{\perp}V=\textstyle\frac{1}{2}V\};\\
\Span\{\cos(\theta_1\pm\theta_2)(JX_1\pm JX_2), \sin(\theta_1\pm\theta_2)(JX_1\pm JX_2)\}=\{V\,:\,\Delta_L^{\perp}V=V\}.
\end{gather*}
\end{lem}

\begin{proof}
This is immediate from Lemma \ref{lem:f.1.espace} and \eqref{eq:L.nDelta}.
\end{proof}

Notice that the fact that there is a 2-dimensional space of harmonic normal vector fields is consistent with the fact that $b^1(L)=2$, and thus the space of harmonic 1-forms is 2-dimensional.

\vspace{-4pt}
\subsection{Entropy} The entropy of an immersion $X:\Sigma^2\to\bC^2$, where $\Sigma^2$ is compact, is defined as in \cite{ColdMin} to be
\begin{equation*}
\lambda(X)=\sup_{(x_0,t_0)\in\bC^2\times\bR^+}\frac{1}{4\pi t_0}\int_{\Sigma}\exp\left(-\frac{|X-x_0|^2}{4t_0}\right) \vol_{\Sigma},
\end{equation*}
where $\vol_{\Sigma}$ is the volume form induced by $X^*g$, and $g$ is the Euclidean metric on $\bC^2$.  It will also be useful to consider the $F$-functional
\begin{equation*}
F(X,x_0,t_0)=\frac{1}{4\pi t_0}\int_{\Sigma}\exp\left(-\frac{|X-x_0|^2}{4t_0}\right) \vol_{\Sigma},
\end{equation*}
so that
\begin{equation*}
\lambda(X)=\sup_{(x_0,t_0)\in\bC^2\times\bR^+}F(X,x_0,t_0).
\end{equation*}

The important properties that the entropy has are summarised as follows (cf.~\cite{ColdMin}).  By rescaled mean curvature flow, we mean the flow where we perform the standard Type I rescaling of mean curvature flow about some space-time point.

\begin{lem}\label{lem:entropy}{$\ $}\\[-3ex]
\begin{itemize}
\item[(a)] The entropy is invariant under translations, dilations and rotations.\\[-2ex]
\item[(b)] The entropy is non-increasing under mean curvature flow and rescaled mean curvature flow.\\[-2ex]
\item[(c)] The critical points of the entropy are the self-shrinkers satisfying
\begin{equation}\label{eq:shrinker2}
H=-\frac{(X-x_0)^{\perp}}{2t_0}
\end{equation}
for some $x_0\in\bC^2$ and $t_0>0$.
\end{itemize}
\end{lem} 

\noindent Notice that if one has a self-shrinker satisfying \eqref{eq:shrinker2}, then by applying a translation and dilation one can ensure that the new shrinker satisfies \eqref{eq:shrinker}.

For a self-shrinker $M$ satisfying \eqref{eq:shrinker} we have that
\begin{equation*}
\lambda(X)=F(X,0,1).
\end{equation*}
It is therefore straightforward to compute the entropy of the Clifford torus.

\begin{lem}\label{lem:L.entropy}
For the Clifford torus $X:L\to\bC^2$, we have
\begin{equation*}
\lambda(X)=\frac{2\pi}{e}=2.311\ldots
\end{equation*}
\end{lem}
\begin{proof} We compute
\begin{align*}
\lambda(X)&=\frac{1}{4\pi}\int_L\exp\left(-\frac{1}{4}|X|^2\right)\vol_L=\frac{1}{4\pi}\int_0^{2\pi}\int_0^{2\pi}2e^{-1}\rd\theta_1\rd\theta_2=\frac{4\pi^2}{2\pi e}=\frac{2\pi}{e},
\end{align*}
where we used $|X|^2=4$ and \eqref{eq:vol}.
\end{proof}

\vspace{-4pt}
\subsection{Second variation}  As we already stated, the first variation of $F$ at $(X,0,1)$ vanishes precisely at self-shrinkers satisfying \eqref{eq:shrinker}.  Therefore, to understand the stability (or otherwise) of the Clifford torus we need to look at the second variation of $F$ at $(X,0,1)$.  This is computed by several authors, e.g.~\cite{AndrewsLiWei, ArezzoSun, LeeLue}, and we specialise their formula to our situation.

\begin{lem}\label{lem:L.secondvar}
The second variation of $F$ at the Clifford torus $L$ at $(X,0,1)$ in a normal direction $V=f_1JX_1+f_2JX_2$ (so fixing $x_0=0$ and $t_0=1$) is given by
\begin{equation}\label{eq:L.secondvar}
4\pi e\left.\frac{\partial^2F}{\partial s^2}\right|_{s=0}=\langle V,(\Delta_L^{\perp}-1)V\rangle_{L^2}=2\langle f_1,(\Delta_L-1)f_1\rangle_{L^2}+2\langle f_2,(\Delta_L-1)f_2\rangle_{L^2}.
\end{equation}
More generally, the second variation of $F$ at $L$ at $(X,0,1)$ in a normal direction $V$ with 
\begin{equation*}
x_0=0,\quad \left.\frac{\partial x_s}{\partial s}\right|_{s=0}=\xi,\quad t_0=1,\quad \left.\frac{\partial t_s}{\partial s}\right|_{s=0}=\tau,
\end{equation*}
is given by
\begin{equation*}
4\pi e\left.\frac{\partial^2F}{\partial s^2}\right|_{s=0}=\langle V,(\Delta_L^{\perp}-1)V+\xi+\tau X\rangle_{L^2}-\frac{1}{2}\|\xi^{\perp}\|_{L^2}^2  -8\pi^2\tau^2.
\end{equation*}
\end{lem}

\begin{proof}
It is shown in \cite[Theorem 3]{LeeLue}, for example, that if $g^{ij}$ denotes the components of the inverse of the induced metric and $A_{ij}$ denotes the components of the second fundamental form on $L$, and we set
\begin{equation*}
\mathcal{L}V=\Delta^{\perp}_LV-\langle A_{ij},V\rangle g^{ki}g^{jl}A_{kl}-\frac{V}{2}+\frac{1}{2}\nabla^{\perp}_{X^{\top}}V.
\end{equation*}
then 
\begin{equation}\label{eq:secondvar2}
4\pi\left.\frac{\partial^2F}{\partial s^2}\right|_{s=0}=\int_L \left(\langle V,\mathcal{L}V\rangle+\langle V,\xi\rangle-\frac{1}{2}|\xi^{\perp}|^2-2\tau\langle H,V\rangle-\tau^2|H|^2\right) e^{-\frac{|X|^2}{4}}\vol_L.
\end{equation}

For the Clifford torus, it is easy to see from \eqref{eq:L.nabla} that the second fundamental form $A$ of $L$ with respect to the basis $X_1,X_2$ is given by:
\begin{equation*}
A_{11}=JX_1,\quad A_{12}=A_{21}=0,\quad A_{22}=JX_2.
\end{equation*}
Therefore, using \eqref{eq:L.metric} we see that
\begin{equation*}
\langle A_{ij},V\rangle g^{ki}g^{jl}A_{kl}= \frac{1}{4}\langle JX_1,V\rangle JX_1+\frac{1}{4}\langle JX_2,V\rangle JX_2=\frac{1}{2}g^{ij}\langle JX_i,V\rangle JX_j=\frac{V}{2}.
\end{equation*}
The fact that $X=X^{\perp}$ implies that $X^{\top}=0$.  Moreover $2H=-X$, $|X|^2=4$ and the volume of $L$ is $8\pi^2$.  The result follows from \eqref{eq:L.nDelta} and \eqref{eq:secondvar2}.
\end{proof}

Lemma \ref{lem:L.secondvar} implies the linearisation of the self-shrinker condition \eqref{eq:shrinker} on $L$ is, up to an overall sign, given by $\Delta_L^{\perp}-1$.  We shall formalise this statement later, but what we mean is that if we consider a normal graph over $L$ which also satisfies \eqref{eq:shrinker}, then to first order the normal vector defining the graph will lie in the kernel of $\Delta_L^{\perp}-1$, which we have described  in Lemma \ref{lem:V.1.espace}.

\section{Group orbits}\label{sec:orbits}

We look at the orbit of the Clifford torus $L$ under various groups, studying those which preserve the Lagrangian or self-shrinker condition, or otherwise.  This will play a crucial role in our later study.

\vspace{-4pt}
\subsection{Dilations} Since $L$ is a self-shrinker, we know that dilations of $L$ are generated by $H$, or equivalently 
\begin{equation}\label{eq:U1}
U_1=-X^{\perp}=JX_1+JX_2.
\end{equation}
Notice that this is a harmonic normal vector field and thus clearly not Hamiltonian.

We can choose another harmonic normal vector field orthogonal to $X^{\perp}$,
\begin{equation}\label{eq:U2}
U_2=JX_1-JX_2,
\end{equation}
and we observe the following from Lemma \ref{lem:V.1.espace}.

\begin{lem}\label{lem:dilations}
For $\Delta_L^{\perp}$ given in \eqref{eq:L.nDelta.rough} and $U_1$, $U_2$ given in \eqref{eq:U1}--\eqref{eq:U2}, we have
\begin{equation*}
\Span\{U_1,U_2\}=\{V\,:\,\Delta_L^{\perp}V=0\}.
\end{equation*}
\end{lem}

If we define a 1-parameter family $\{L_{\delta_s}:s\in\bR\}$ of Lagrangians by
\begin{equation*}
L_{\delta_s}=\left\{\frac{\sqrt{2}}{\sqrt{\cosh 2s}}(e^{-s+i\theta_1},e^{s+i\theta_2})\,:\,\theta_1,\theta_2\in\bR\right\},
\end{equation*}
we see that $L_{\delta_s}\subseteq\mathcal{S}^3(2)$ for all $s$, $L_{\delta_0}=L$ and we can calculate the variation vector field 
\begin{equation*}
\left.\frac{\partial L_{\delta_s}}{\partial s}\right|_{s=0}=JX_1-JX_2=U_2.
\end{equation*}
We shall see that $L_{\delta_s}$ define Lagrangian variations for which the Clifford torus is unstable under the flow, but we know these lie in  different Hamiltonian isotopy classes to  $L$ for $s\neq 0$ as $U_2$ is not Hamiltonian.  

\vspace{-4pt}
\subsection{Translations} Translations clearly preserve the class of self-shrinkers and preserve the Lagrangian condition.  The translations on $\bC^2$ are generated by the vectors
\begin{equation*}
(1,0),\quad (i,0),\quad (0,1),\quad (0,i). 
\end{equation*}
We can restrict these vector fields to the Clifford torus $L$ and we may compute
\begin{align*}
\langle (1,0),JE_1\rangle&=-\Ree(e^{-i\theta_1})=-\cos\theta_1, &
\langle (1,0),JE_2\rangle&=0,\\
\langle (i,0),JE_1\rangle&=-\Ree(ie^{-i\theta_1})=-\sin\theta_1, &
\langle (i,0),JE_2\rangle&=0,\\
\langle (0,1),JE_1\rangle&=0, &
\langle (0,1),JE_2\rangle&=-\Ree(e^{-i\theta_2})=-\cos\theta_2,\\
\langle (0,i),JE_1\rangle&=0, &
\langle (0,i),JE_2\rangle&=-\Ree(ie^{-i\theta_2})=-\sin\theta_2.
\end{align*}
Therefore,
\begin{align}
(1,0)^{\perp}&=-\cos\theta_1JE_1=-\sqrt{2}J\nabla(\sin\theta_1),
&
(i,0)^{\perp}&=-\sin\theta_1JE_1=\sqrt{2}J\nabla(\cos\theta_1),\label{eq:trans1}\\
(0,1)^{\perp}&=-\cos\theta_2JE_2=-\sqrt{2}J\nabla(\sin\theta_2),
&
(0,i)^{\perp}&=-\sin\theta_2JE_2=\sqrt{2}J\nabla(\cos\theta_2),\label{eq:trans2}
\end{align}
which are manifestly Hamiltonian.  Moreover, we have the following.

\begin{lem}\label{lem:translations}
For $\Delta_L$ given in \eqref{eq:L.Delta}, we have
\begin{equation*}
\Span\{(1,0)^{\perp},(i,0)^{\perp},(0,1)^{\perp},(0,i)^{\perp}\}=
\{J\nabla f\,:\,\Delta_L f=\textstyle\frac{1}{2} f\}.
\end{equation*}
\end{lem}

\begin{proof}
This is immediate from Lemma \ref{lem:f.1.espace} and \eqref{eq:trans1}--\eqref{eq:trans2}
\end{proof}

\vspace{-4pt}
\subsection{Unitary transformations}
We know that the unitary group $\UU(2)$ on $\bC^2$ is the intersection of the rigid isometry group $\SO(4)$ on $\bR^4=\bC^2$ with the Hamiltonian diffeomorphism group on $\bC^2$.   Therefore, the orbit of $L$ under $\UU(2)$ consists of Lagrangian self-shrinkers satisfying \eqref{eq:shrinker}, and the orbit of $L$ under $\SO(4)$ consists of (not necessarily Lagrangian) self-shrinkers satisfying \eqref{eq:shrinker}.  

The orbit of $L$ under $\UU(2)$ is 2-dimensional and the orbit of $L$ under $\SO(4)$ is 4-dimensional, since the stabilizer of $L$ in each case is the maximal torus in $\UU(2)$:
\begin{equation*}
\left\{\left(\begin{array}{cc} e^{i\phi_1} & 0 \\ 0 & e^{i\phi_2}\end{array}\right)\,:\,\phi_1,\phi_2\in\bR\right\}.
\end{equation*}
The maximal torus is generated by the matrices
\begin{equation*}
\left(\begin{array}{cc} i & 0 \\ 0 & 0\end{array}\right)\quad\text{and}\quad
\left(\begin{array}{cc} 0 & 0 \\ 0 &  i\end{array}\right),
\end{equation*}
leading to vector fields on $\bC^2$
\begin{equation*}
\left(\begin{array}{c} iz_1\\ 0 \end{array}\right)\quad\text{and}\quad \left(\begin{array}{c} 0 \\ iz_2\end{array}\right),
\end{equation*}
 generating one-parameter subgroups in $\UU(2)\subseteq\SO(4)$.  Notice that these vector fields restricted to $L$ are just $X_1$ and $X_2$, so their projection to the normal space of $L$ is zero.  (Here, and throughout, we will not distinguish row vectors and column vectors.)
 
 We can find a complementary (in fact, orthogonal) subspace of the Lie algebra of $\UU(2)$ to the maximal torus, spanned by the matrices
 \begin{equation*}
 \left(\begin{array}{cc} 0 & -1\\ 1 & 0\end{array}\right)\quad\text{and}\quad 
 \left(\begin{array}{cc} 0 & i \\ i & 0\end{array}\right).
\end{equation*}
The corresponding vector fields on $\bC^2$ are
\begin{equation}\label{eq:u2.vfields}
\left(\begin{array}{c} -z_2\\ z_1\end{array}\right)\quad\text{and}\quad \left(\begin{array}{c} iz_2\\ iz_1\end{array}\right), 
\end{equation}
whose restrictions to $L$ are just
\begin{equation*}
Y_1=\sqrt{2}(-e^{i\theta_2},e^{i\theta_1})\quad\text{and}\quad Y_2=\sqrt{2}(ie^{i\theta_2},ie^{i\theta_1}).
\end{equation*}
We quickly see that
\begin{align*}
\langle Y_1,JE_1\rangle&=\Ree(\sqrt{2}e^{i(\theta_2-\theta_1)})=\sqrt{2}\cos(\theta_1-\theta_2),\\
\langle Y_1, JE_2\rangle &=\Ree(-\sqrt{2}e^{i(\theta_1-\theta_2)})=-\sqrt{2}\cos(\theta_1-\theta_2),\\
\langle Y_2,JE_1\rangle &=\Ree(-\sqrt{2}ie^{i(\theta_2-\theta_1)})=-\sqrt{2}\sin(\theta_1-\theta_2),\\
\langle Y_2,JE_2\rangle &=\Ree(-\sqrt{2}ie^{i(\theta_1-\theta_2)})=\sqrt{2}\sin(\theta_1-\theta_2).
\end{align*}
Hence,
\begin{align}\label{eq:Y1perp}
Y_1^{\perp}&=\cos(\theta_1-\theta_2)(JX_1-JX_2)=2J\nabla(\sin(\theta_1-\theta_2)),\\
 Y_2^{\perp}&=-\sin(\theta_1-\theta_2)(JX_1-JX_2)=2J\nabla(\cos(\theta_1-\theta_2)),\label{eq:Y2perp}
\end{align}
which are manifestly Hamiltonian.

 If we consider the action of the matrix 
\begin{equation}\label{eq:equiv.rotate}
\frac{1}{\sqrt{2}}\left(\begin{array}{cc} 1 & 1 \\ -i & i\end{array}\right)\in \UU(2)
\end{equation}
on $L$ we see that we obtain the Lagrangian self-shrinker
\begin{align}
L'&=\{(e^{i\theta_1}+e^{i\theta_2},-ie^{i\theta_1}+ie^{i\theta_2})\,:\,\theta_1,\theta_2\in\bR\}\nonumber\\
&=\big\{\big( (\cos\theta_1+\cos\theta_2)+i(\sin\theta_1+\sin\theta_2),(\sin\theta_1-\sin\theta_2)-i(\cos\theta_1-\cos\theta_2)\big)\,:\,\theta_1,\theta_2\in\bR\big\}\nonumber\\
&=\Big\{\Big(2\cos\left(\frac{\theta_1+\theta_2}{2}\right)\cos\left(\frac{\theta_1-\theta_2}{2}\right)+2i\sin\left(\frac{\theta_1+\theta_2}{2}\right)\cos\left(\frac{\theta_1-\theta_2}{2}\right),\nonumber\\
&\quad \quad \;\,\, 2\cos\left(\frac{\theta_1+\theta_2}{2}\right)\sin\left(\frac{\theta_1-\theta_2}{2}\right)+2i\sin\left(\frac{\theta_1+\theta_2}{2}\right)\sin\left(\frac{\theta_1-\theta_2}{2}\right)\,:\,\theta_1,\theta_2\in\bR\Big\}\nonumber\\
&=\{2(e^{i\phi}\cos\rho,e^{i\phi}\sin\rho)\,:\,\phi,\rho\in\bR\},\label{eq:L'}
\end{align}
where we set
\begin{equation}\label{eq:angles}
\phi=\frac{\theta_1+\theta_2}{2}\quad\text{and}\quad\rho=\frac{\theta_1-\theta_2}{2}.
\end{equation}
In this way, we can view the Clifford torus as an $\mathcal{S}^1$-invariant Lagrangian of the form
\begin{equation*}
\{\gamma(\phi)(\cos\rho,\sin\rho)\,:\,\phi,\rho\in\bR\}
\end{equation*}
for a curve $\gamma$ in $\bC$.  In the case of the Clifford torus, the curve in $\bC$ in question is just
\begin{equation*}
\gamma(\phi)=2e^{i\phi},
\end{equation*}
a circle of radius $2$.
 
\vspace{-4pt}
\subsection{Hamiltonian group orbits}

From the perspective in \eqref{eq:L'} it is clear that we can act by the linear Hamiltonian group on $\bC=\bR^2$, i.e.~$\SL(2,\bR)$, on the curve $\gamma(\phi)=2e^{i\phi}$ to obtain Lagrangians Hamiltonian isotopic to $L'$, and thus $L$.  
The stabilizer of $\gamma$ in $\SL(2,\bR)$ is $\SO(2)$, so the orbit of $\gamma$ under $\SL(2,\bR)$ is 2-dimensional.  Moreover, if $\tilde{\gamma}$ is in the $\SL(2,\bR)$ orbit of $\gamma$, then the corresponding Lagrangian
\begin{equation}\label{eq:equiv2}
\tilde{L}'=\{\tilde{\gamma}(\phi)(\cos\rho,\sin\rho)\,:\,\phi,\rho\in\bR\}
\end{equation}
lies in the $\SO(4,\bR)$ orbit of $L'$, and thus $L$, if and only if $\tilde{\gamma}=\gamma$ up to reparametrisation, which is if and only if $\tilde{L}'=L'$.  

We can choose two one-parameter subgroups of $\SL(2,\bR)$ which, together with  $\SO(2)$, enable us to generate $\SL(2,\bR)$: for example, we can take
\begin{equation*}
A_s=\left(\begin{array}{cc} e^s &  0\\ 0 & e^{-s}\end{array}\right)\quad\text{and}\quad B_s=\left(\begin{array}{cc} \cosh s & \sinh s\\ \sinh s & \cosh s\end{array}\right), 
\end{equation*}
for $s\in\bR$, so that $\{A_s\,:\,s\in\bR\}$ and $\{B_s\,:\,s\in\bR\}$ are our one-parameter subgroups.  We see that, identifying $\bR^2=\bC$, we have
\begin{equation*}
A_s\gamma(\phi)=2e^s\cos\phi+2ie^{-s}\sin\phi
\end{equation*}
and
\begin{align*}
B_s\gamma(\phi)&=2(\cosh s\cos\phi+\sinh s\sin\phi)+2i(\sinh s\cos\phi+\cosh s\sin\phi)\nonumber\\
&=2(\cosh s e^{i\phi}+i\sinh s e^{-i\phi}).
\end{align*}

Taking $\tilde{\gamma}=A_s\gamma$ or $\tilde{\gamma}=B_s\gamma$ in \eqref{eq:equiv2} leads to the following Lagrangians Hamiltonian isotopic to $L'$ (and $L$) which only lie in the $\SO(4)$-orbit of $L$ for $s=0$:
\begin{align*}
L'_{A_s}&=\{2(e^s\cos\phi+ie^{-s}\sin\phi)(\cos\rho,\sin\rho)\,:\,\phi,\rho\in\bR\},\\
L'_{B_s}&=\{2(\cosh s e^{i\phi}+i\sinh se^{-i\phi})(\cos\rho,\sin\rho)\,:\,\phi,\rho\in\bR\}.
\end{align*}
(The statement about not lying in the $\SO(4)$-orbit for $s\neq 0$, as well as being clear by inspection, also follows from \eqref{eq:AB.matrices} below.)
Acting by the inverse of the unitary matrix in \eqref{eq:equiv.rotate}, namely
\begin{equation*}
\frac{1}{\sqrt{2}}\left(\begin{array}{cc} 1 & i \\ 1 & -i \end{array}\right)\in\UU(2),
\end{equation*}
 on $L'_{A_s}$ and $L'_{B_s}$, we obtain Lagrangians $L_{A_s}$ and $L_{B_s}$ Hamiltonian isotopic to $L$, which obviously still only lie in the $\SO(4)$ orbit of $L$ for $s=0$.  Explicitly, we see that
\begin{align*}
L_{A_s}&=\{\sqrt{2}(e^s\cos\phi+ie^{-s}\sin\phi)(e^{i\rho},e^{-i\rho})\,:\,\phi,\rho\in\bR\},\\
L_{B_s}&=\{\sqrt{2}(\cosh s e^{i\phi}+i\sinh se^{-i\phi})(e^{i\rho},e^{-i\rho})\,:\,\phi,\rho\in\bR\}.
\end{align*}
Substituting back for $\phi,\rho$ in terms of $\theta_1,\theta_2$ via \eqref{eq:angles} we compute:
\begin{align*}
2(e^s\cos\phi+ie^{-s}\sin\phi)e^{i\rho}&=
e^s(\cos\theta_1+\cos\theta_2)+e^{-s}(\cos\theta_1-\cos\theta_2)\nonumber\\
&\quad +ie^s(\sin\theta_1-\sin\theta_2)+ie^{-s}(\sin\theta_1+\sin\theta_2)\nonumber\\
&=2\cosh se^{i\theta_1}+2\sinh se^{-i\theta_2};\\
2(e^s\cos\phi+ie^{-s}\sin\phi)e^{-i\rho}
&=e^s(\cos\theta_1+\cos\theta_2)+e^{-s}(\cos\theta_2-\cos\theta_1)\nonumber\\
&\quad+ ie^s(\sin\theta_2-\sin\theta_1)+ie^{-s}(\sin\theta_1+\sin\theta_2)\nonumber\\
&=2\sinh se^{-i\theta_1}+2\cosh se^{i\theta_2};\\
2(\cosh se^{i\phi}+i\sinh se^{-i\phi})e^{i\rho}&=
2\cosh se^{i\theta_1}+2i\sinh s e^{-i\theta_2};\\
2(\cosh se^{i\phi}+i\sinh se^{-i\phi})e^{-i\rho}&=2i\sinh se^{-i\theta_1}+2\cosh s e^{i\theta_2}.
\end{align*}
We can thus rewrite 
\begin{align}
L_{A_s}&=\{\sqrt{2}(\cosh se^{i\theta_1}+\sinh se^{-i\theta_2},\sinh se^{-i\theta_1}+\cosh s e^{i\theta_2})\,:\,\theta_1,\theta_2\in\bR\},\label{eq:LAs}\\
L_{B_s}&=\{\sqrt{2}(\cosh se^{i\theta_1}+i\sinh se^{-i\theta_2},i\sinh se^{-i\theta_1}+\cosh s e^{i\theta_2})\,:\,\theta_1,\theta_2\in\bR\}.\label{eq:LBs}
\end{align}

Notice that the variation vector fields for $L_{A_s}$ and $L_{B_s}$ at $s=0$ are given by
\begin{equation*}
V_A=\sqrt{2}(e^{-i\theta_2},e^{-i\theta_1})\quad\text{and}\quad V_B=\sqrt{2}(ie^{-i\theta_2},ie^{-i\theta_1}).
\end{equation*}
We may compute that
\begin{align*}
\langle V_A,JE_1\rangle&=\Ree(-\sqrt{2}e^{-i(\theta_2+\theta_1)})=-\sqrt{2}\cos(\theta_1+\theta_2),\\
\langle V_A,JE_2\rangle &=\Ree(-\sqrt{2}e^{-i(\theta_1+\theta_2)})=-\sqrt{2}\cos(\theta_1+\theta_2),\\
\langle V_B,JE_1\rangle &=\Ree(-\sqrt{2}ie^{-i(\theta_2+\theta_1)})=-\sqrt{2}\sin(\theta_1+\theta_2),\\
\langle V_B,JE_2\rangle &=\Ree(-\sqrt{2}ie^{-i(\theta_1+\theta_2)})=-\sqrt{2}\sin(\theta_1+\theta_2).
\end{align*}
Hence,
\begin{align}
V_A^{\perp}&=-\cos(\theta_1+\theta_2)(JX_1+JX_2)=-2J\nabla(\sin(\theta_1+\theta_2)),\label{eq:VAperp}\\
V_B^{\perp}&=-\sin(\theta_1+\theta_2)(JX_1+JX_2)=2J\nabla(\cos(\theta_1+\theta_2)),\label{eq:VBperp}
\end{align}
which are both clearly Hamiltonian, as we knew.

It is worth noting the following, which follows immediately from Lemma \ref{lem:f.1.espace}.

\begin{lem}\label{lem:f.1.espace2}
For  the normal vector fields given in \eqref{eq:Y1perp}, \eqref{eq:Y2perp}, \eqref{eq:VAperp}, \eqref{eq:VBperp}, we have
\begin{equation*}
\Span\{Y_1^{\perp},Y_2^{\perp},V_A^{\perp},V_B^{\perp}\}=\{J\nabla f\,:\,\Delta_L f=f\}.
\end{equation*}
\end{lem}

Observe that the matrices defining the vector fields on $\bR^4=\bC^2$, which generate the one-parameter subgroups of transformations defining the families $L_{A_s}$ and $L_{B_s}$, are given by
 \begin{equation}\label{eq:AB.matrices}
 \left(\begin{array}{cccc} 0 & 0 & 1 & 0 \\ 0 & 0 & 0 & -1\\ 1 & 0 & 0 & 0\\ 0 & -1 &0 &0  \end{array} \right) 
\quad\text{and}\quad 
\left( \begin{array}{cccc} 0 & 0 & 0 & 1 \\ 0 & 0 & 1 & 0\\ 0 & 1 & 0 & 0\\ 1 & 0 &0 &0  \end{array}\right).
\end{equation}
These matrices lie in $\mathfrak{sp}(4,\bR)$, the Lie algebra of the symplectic group on $\bR^4$, but clearly do not lie in $\mathfrak{so}(4)$ (and thus do not lie in $\mathfrak{u}(2)$).

\vspace{-4pt}
\subsection{Rotations} 
We have so far focused on the Clifford torus $L$ as a Lagrangian self-shrinker, but we now want to understand its character just as a self-shrinker.  
For this, we first need to identify the rotations in $\SO(4)$ which do not arise from $\UU(2)$.  At the Lie algebra level (i.e.~in $\mathfrak{so}(4)$), we can span this 2-dimensional space with the following matrices:
\begin{equation*}
 \left(\begin{array}{cccc} 0 & 0 & -1 & 0 \\ 0 & 0 & 0 & 1\\ 1 & 0 & 0 & 0\\ 0 & -1 &0 &0  \end{array} \right) 
\quad\text{and}\quad 
\left( \begin{array}{cccc} 0 & 0 & 0 & -1 \\ 0 & 0 & -1 & 0\\ 0 & 1 & 0 & 0\\ 1 & 0 &0 &0  \end{array}\right).
\end{equation*}
This yields corresponding vector fields on $\bR^4=\bC^2$,
\begin{equation*}
\left(\begin{array}{c}
-x_2 \\
y_2 \\
x_1 \\
-y_1 
\end{array}
\right)
=\left(\begin{array}{c} -\overline{z_2}\\ \overline{z_1}\end{array} \right) 
\quad\text{and}\quad 
\left(\begin{array}{c}
-y_2 \\
-x_2 \\
y_1 \\
x_1 
\end{array}\right)=\left(\begin{array}{c} -i\overline{z_2}\\ i\overline{z_1}\end{array} \right) ,
\end{equation*}
generating one-parameter subgroups in $\SO(4)$.  Their restrictions to $L$ are
\begin{equation*}
 Y_3=\sqrt{2}(-e^{-i\theta_2},e^{-i\theta_1})\quad\text{and}\quad Y_4=\sqrt{2}(-ie^{-i\theta_2},ie^{-i\theta_1}).
\end{equation*}
As before, we may compute
\begin{align*}
 \langle Y_3,JE_1\rangle &=\Ree(\sqrt{2}e^{-i(\theta_2+\theta_1)})=\sqrt{2}\cos(\theta_1+\theta_2),\\
 \langle Y_3,JE_2\rangle &=\Ree(-\sqrt{2}e^{-i(\theta_1+\theta_2)})=-\sqrt{2}\cos(\theta_1+\theta_2),\\
 \langle Y_4,JE_1\rangle &=\Ree(\sqrt{2}ie^{-i(\theta_2+\theta_1)})=\sqrt{2}\sin(\theta_1+\theta_2),\\
 \langle Y_4,JE_2\rangle &=\Ree(-\sqrt{2}ie^{-i(\theta_1+\theta_2)})=-\sqrt{2}\sin(\theta_1+\theta_2).
\end{align*}
Thus, we have
\begin{equation}\label{eq:Y3perp.Y4perp}
 Y_3^{\perp}=\cos(\theta_1+\theta_2)(JX_1-JX_2)\quad\text{and}\quad Y_4^{\perp}=\sin(\theta_1+\theta_2)(JX_1-JX_2).
\end{equation}
Notice here that these vector fields are \emph{not} Hamiltonian, again as we would expect.

\vspace{-4pt}
\subsection{Further group orbits}
Finally, we consider the following $2\times 2$ complex (in fact, Hermitian) matrices at the Lie algebra level (i.e.~they lie in the Lie algebra of $\SL(2,\bC)$):
\begin{equation*}
\left(\begin{array}{cc} 0 & 1 \\ 1 &  0 \end{array}\right)
\quad\text{and}\quad
\left(\begin{array}{cc} 0 & -i\\ i & 0\end{array}\right),
\end{equation*}
which are identified with the following real $4\times 4$ matrices
\begin{equation*}
 \left(\begin{array}{cccc} 0 & 0 & 1 & 0 \\ 0 & 0 & 0 & 1\\ 1 & 0 & 0 & 0\\ 0 & 1 &0 &0  \end{array} \right) 
\quad\text{and}\quad 
\left( \begin{array}{cccc} 0 & 0 & 0 & 1 \\ 0 & 0 & -1 & 0\\ 0 & -1 & 0 & 0\\ 1 & 0 &0 &0  \end{array}\right),
\end{equation*}
that do not lie in $\mathfrak{so}(4)$ or in $\mathfrak{sp}(4,\bR)$.
These matrices yield corresponding vector fields on $\bC^2$:
\begin{equation*}
\left(\begin{array}{c} z_2 \\ z_1\end{array}\right) \quad\text{and}
\quad \left(\begin{array}{c} -iz_2\\ iz_1\end{array}\right),
\end{equation*}
which one should compare to \eqref{eq:u2.vfields}.  These vector fields generate one-parameter groups given by $\{C_s\,:\,s\in\bR\}$ and $\{D_s\,:\,s\in\bR\}$, where
\begin{equation*}
C_s=\left(\begin{array}{cc} \cosh s & \sinh s\\ 
\sinh s & \cosh s \end{array}\right)\quad\text{and}\quad
D_s=\left(\begin{array}{cc} \cosh s & -i\sinh s\\ 
i\sinh s & \cosh s \end{array}\right).
\end{equation*}
The orbits of $L$ under the action of these one-parameter groups yield the following real surfaces in $\bC^2$, which lie in the $\SO(4)$-orbit of $L$ only for $s=0$:
\begin{align*}
L_{C_s}&=\{\sqrt{2}(\cosh s e^{i\theta_1}+\sinh s e^{i\theta_2},
\sinh s e^{i\theta_1}+\cosh s e^{i\theta_2})\,:\,\theta_1,\theta_2\in\bR\},\\
L_{D_s}&=\{\sqrt{2}(\cosh s e^{i\theta_1}-i\sinh s e^{i\theta_2},
i\sinh s e^{i\theta_1}+\cosh s e^{i\theta_2})\,:\,\theta_1,\theta_2\in\bR\}.
\end{align*}
(These formulae should be compared to $L_{A_s}$ and $L_{B_s}$ in \eqref{eq:LAs}--\eqref{eq:LBs}.)  

The variation vector fields of $L_{C_s}$ and $L_{D_s}$ at $s=0$ are given by
\begin{equation*}
V_C=\sqrt{2}(e^{i\theta_2},e^{i\theta_1})\quad\text{and}\quad 
V_D=\sqrt{2}(-ie^{i\theta_2},ie^{i\theta_1}).
\end{equation*}
We may compute that
\begin{align*}
\langle V_C,JE_1\rangle&=\Ree(-\sqrt{2}e^{i(\theta_2-\theta_1)})=-\sqrt{2}\cos(\theta_1-\theta_2),\\
\langle V_C,JE_2\rangle &=\Ree(-\sqrt{2}e^{i(\theta_1-\theta_2)})=-\sqrt{2}\cos(\theta_1-\theta_2),\\
\langle V_D,JE_1\rangle &=\Ree(\sqrt{2}ie^{i(\theta_2-\theta_1)})=\sqrt{2}\sin(\theta_1-\theta_2),\\
\langle V_D,JE_2\rangle &=\Ree(-\sqrt{2}ie^{i(\theta_1-\theta_2)})=\sqrt{2}\sin(\theta_1-\theta_2).
\end{align*}
Hence,
\begin{align}\label{eq:VCperp.VDperp}
V_C^{\perp}=-\cos(\theta_1-\theta_2)(JX_1+JX_2)\quad\text{and}\quad
V_D^{\perp}=\sin(\theta_1-\theta_2)(JX_1+JX_2),
\end{align}
which should be compared to \eqref{eq:VAperp}--\eqref{eq:VBperp}, and which are manifestly \emph{not} Hamiltonian, as we would expect.

We now observe the following by Lemma \ref{lem:V.1.espace}.

\begin{lem}\label{lem:V.1.espace.2} For the normal vector fields on $L$ given in \eqref{eq:Y1perp}, \eqref{eq:Y2perp}, \eqref{eq:VAperp}, \eqref{eq:VBperp}, \eqref{eq:Y3perp.Y4perp}, \eqref{eq:VCperp.VDperp}, we have 
\begin{equation*}
\Span\{Y_1^{\perp},Y_2^{\perp},Y_3^{\perp},Y_4^{\perp},
V_A^{\perp},V_B^{\perp},V_C^{\perp},V_D^{\perp}\}=
\{V\,:\,\Delta_L^{\perp}V=V\}.
\end{equation*}
\end{lem}

\section{Hamiltonian instability}\label{sec:instability}

It is known that the Clifford torus is $F$-unstable under Lagrangian variations \cite[Theorem 8]{LeeLue} but $F$-stable under Hamiltonian variations (this is claimed in \cite{LeeLue,LiZhang}, even though \cite[Main Theorem 6]{ArezzoSun} seems erroneously to claim the contrary).  We will verify these claims explicitly and show more: that the Clifford torus is entropy unstable under Hamiltonian variations.  

From this we will prove that the Clifford torus is unstable under Lagrangian mean curvature flow under $C^{\infty}$-small Hamiltonian variations (even $\UU(1)$-equivariant and graphical).

\vspace{-4pt}
\subsection{Second variation}  We first verify the Lagrangian instability result in \cite[Theorem 8]{LeeLue}.

\begin{lem}
In the direction of the Lagrangian variation $U_2=JX_1-JX_2$ of the Clifford torus $L$, the second variation of $F$ is strictly negative.  Thus, $L$ is Lagrangian $F$-unstable.
\end{lem}

\begin{proof}
We have that $\Delta_L^{\perp}U_2=0$ by Lemma \ref{lem:V.1.espace}.  We know that $U_2$ is orthogonal to $X=JX_1+JX_2$ by \eqref{eq:L.metric.0}, and 
$U_2$ is orthogonal to the restriction of any constant vector to $L$ by Lemmas \ref{lem:V.1.espace} and \ref{lem:translations}.  The result follows
 from taking $V=U_2$ in Lemma \ref{lem:L.secondvar}.  
\end{proof}

\noindent This Lagrangian variation corresponds to ``squashing'' one geodesic circle direction in the torus whilst ``expanding'' the orthogonal geodesic circle direction.  We reiterate that this variation is \emph{not} Hamiltonian.

We know, by Lemmas \ref{lem:V.1.espace}, \ref{lem:L.secondvar}  and \ref{lem:translations}, that the second variation of $F$ will be negative in the direction of any translation since they have eigenvalue $\frac{1}{2}$ for $\Delta_L^{\perp}$.  These transformations are also Hamiltonian so they give unstable Hamiltonian directions for $F$.  However, we now show that these are the only unstable directions, verifying the Hamiltonian $F$-stability claimed in \cite{LeeLue,LiZhang}.

\begin{lem}\label{lem:L.HamFstable}
In the direction of any Hamiltonian variation of the Clifford torus $L$ orthogonal to the translations, the second variation of $F$ is non-negative.  Thus, $L$ is Hamiltonian $F$-stable.
\end{lem}

\begin{proof}
If $J\nabla f$ is a Hamiltonian vector field orthogonal to the translations, then $f$ must lie in the span of the eigenspaces of $\Delta_L$ of eigenvalue greater than or equal to $1$ by Lemmas \ref{lem:f.1.espace} and \ref{lem:translations}.  The result follows from Lemma \ref{lem:L.secondvar}.
\end{proof}

This result fits well with the following \cite{Oh}.

\begin{thm}\label{thm:L.localmin.vol}
The Clifford torus $L$ is a local minimum for volume under Hamiltonian variations in $\bC^2$.
\end{thm}

We finally characterise the kernel of the second variation under Hamiltonian variations.

\begin{lem}\label{lem:second.var.zero}
The Hamiltonian vector fields $J\nabla f$ which give directions for which the second variation of $F$ at the Clifford torus $L$ is zero are those where $\Delta_Lf=f$, described in Lemma \ref{lem:f.1.espace2}.
\end{lem}

By Lemma \ref{lem:f.1.espace2}, two of these variations arise from unitary transformations and therefore are integrable directions, in the sense that the $F$-functional is constant under these transformations.  However, we need to analyse further the other two directions, as we only currently know  that the $F$-functional is non-decreasing in these directions to second order.

\vspace{-4pt}
\subsection{Entropy}  We now return to the Lagrangians $L_{A_s}$ introduced in \eqref{eq:LAs}.  By Lemmas \ref{lem:f.1.espace2} and \ref{lem:second.var.zero}, they are generated at $s=0$ by a Hamiltonian vector field for which the second variation at $s=0$ is zero. If we let $X(s)$ be the position vector of $L_{A_s}$, we wish to compute the value of the $F$-functional at $(X(s),0,1)$ for $s$ near $0$ to compare it to its value at $(X,0,1)$ as $X(0)=X$.  

After that, we wish to estimate $F$ at $(X(s),x_0,t_0)$ for $s$ near $0$ and $(x_0,t_0)$ near $(0,1)$.  With this information, we wish to compare the value of the entropy $\lambda(X(s))$ relative to $\lambda(X)$, computed in Lemma \ref{lem:L.entropy}.

We begin by showing that the $F$-functional, centred at $(x_0,t_0)=(0,1)$ strictly decreases along the family $L_{A_s}$ for $s$ near $0$.

\begin{prop}\label{prop:F.expand}
For $s$ near $0$, we have that
\begin{equation*}
F(X(s),0,1)-F(X(0),0,1)=-\frac{4\pi}{9 e}s^6+O(s^8).
\end{equation*}
Hence $F(X(s),0,1)$ has a strict local maximum at $s=0$.
\end{prop}

\begin{proof}
We start by recalling that
\begin{equation} \label{eq:Xs}
X(s)=\sqrt{2}(\cosh se^{i\theta_1}+\sinh se^{-i\theta_2},\sinh se^{-i\theta_1}+\cosh s e^{i\theta_2}).
\end{equation}
Hence,
\begin{align*}
|X(s)|^2&=2\big((\cosh s\cos\theta_1+\sinh s\cos\theta_2)^2+(\cosh s\sin\theta_1-\sinh s\sin\theta_2)^2\nonumber\\
&\qquad+(\sinh s\cos\theta_1+\cosh s\cos\theta_2)^2
+(-\sinh s\sin\theta_1+\cosh s\sin\theta_2)^2\big)\nonumber\\
&=2\big(2(\cosh^2 s+\sinh^2 s)+4\sinh s\cosh s\cos(\theta_1+\theta_2)\big)\nonumber\\
&=4\cosh 2s+4\sinh 2s\cos(\theta_1+\theta_2). 
\end{align*}
We have two tangent vector fields on $L_{A_s}$:
\begin{equation*} 
X_1(s)=\sqrt{2}(i\cosh se^{i\theta_1},-i\sinh se^{-i\theta_1})
\quad\text{and}\quad
X_2(s)=\sqrt{2}(-i\sinh se^{-i\theta_2},i\cosh se^{i\theta_2}).
\end{equation*}
We see that
\begin{gather*}
|X_1(s)|^2=|X_2(s)|^2=2(\cosh^2s+\sinh^2s)=2\cosh 2s,\\
\langle X_1(s),X_2(s)\rangle =-2\cosh s\sinh s\Ree(e^{i(\theta_1+\theta_2)}+e^{-i(\theta_1+\theta_2)})=-2\sinh 2s\cos(\theta_1+\theta_2),
\end{gather*}
so the induced metric on $L_{A_s}$ is
\begin{equation*} 
2\cosh 2s (\rd\theta_1^2+\rd\theta_2^2)-2\sinh 2s\cos(\theta_1+\theta_2)\rd\theta_1\rd\theta_2.
\end{equation*}
Therefore, the volume form on $L_{A_s}$ is
\begin{align} \label{eq:vols}
\vol(s)&=2\sqrt{\cosh^22s-\sinh^22s\cos^2(\theta_1+\theta_2)}\,\rd\theta_1\wedge\rd\theta_2.
\end{align}
Hence, we deduce that
\begin{align}\label{eq:F.expand}
&F(X(s),0,1)=\frac{1}{2\pi e} \int_0^{2\pi} \int_0^{2\pi} 
I(s)\rd\theta_1\rd\theta_2,
\end{align}
where
\begin{equation}\label{eq:I}
I(s)=\sqrt{\cosh^22s-\sinh^22s\cos^2(\theta_1+\theta_2)}
e^{1-\cosh 2s-\sinh 2s\cos(\theta_1+\theta_2)}.
\end{equation}
This is clearly a real analytic function of $s$ and we may then compute its power series expansion about $s=0$.  We see that since changing $s$ to $-s$ in \eqref{eq:I} can be accounted for by translating $\theta_1+\theta_2$ to $\theta_1+\theta_2+\pi$, the power series of \eqref{eq:F.expand} will be even in $s$.  Equivalently, we notice that the odd powers of $s$ in the expansion of $I(s)$ will be linear combinations of odd powers of $\cos(\theta_1+\theta_2)$ and so will integrate to $0$ in \eqref{eq:F.expand}.  We see this explicitly when we compute the first terms of the power series, by calculating:
\begin{gather}
I(0)=1,\quad  I'(0)=-2\cos(\theta_1+\theta_2),\quad  I''(0)=0,\label{eq:I.expand1}\\ 
\frac{I^{(3)}(0)}{3!}=-\frac{4}{3}\cos(\theta_1+\theta_2)+\frac{8}{3}\cos^3(\theta_1+\theta_2),\label{eq:I.expand3}\\
\frac{I^{(4)}(0)}{4!}=-2+8\cos^2(\theta_1+\theta_2)-\frac{16}{3}\cos^4(\theta_1+\theta_2),\label{eq:I.expand4}\\
\frac{I^{(5)}(0)}{5!}=\frac{56}{15}\cos(\theta_1+\theta_2)-\frac{32}{3}\cos^3(\theta_1+\theta_2)+\frac{32}{5}\cos^5(\theta_1+\theta_2),\label{eq:I.expand5}\\
\frac{I^{(6)}(0)}{6!}=\frac{4}{3}-\frac{32}{3}\cos^2(\theta_1+\theta_2)+\frac{160}{9}\cos^4(\theta_1+\theta_2)-\frac{416}{45}\cos^6(\theta_1+\theta_2).\label{eq:I.expand6}
\end{gather}
Notice that \eqref{eq:I.expand1} implies that, for $s$ near $0$,
\begin{equation*}
F(X(s),0,1)=\frac{2\pi}{e}+O(s^3)=F(X,0,1)+O(s^3),
\end{equation*}
which is consistent with the fact that $X$ is a critical point for $F$ and that the second variation is zero in the direction $\frac{\partial X(s)}{\partial s}|_{s=0}$.  As already observed, the terms in \eqref{eq:I.expand3} and \eqref{eq:I.expand5} integrate to $0$.  It is also elementary to see from \eqref{eq:I.expand4} and \eqref{eq:I.expand6} that:
\begin{equation*}
\int_0^{2\pi} \int_0^{2\pi} 
\frac{I^{(4)}(0)}{4!}\rd\theta_1\rd\theta_2=0\quad\text{and}
\quad \int_0^{2\pi} \int_0^{2\pi} 
\frac{I^{(6)}(0)}{6!}\rd\theta_1\rd\theta_2=-\frac{8}{9}\pi^2.
\end{equation*}
The result now follows.
\end{proof}

We now consider the value of the $F$-functional for $L_{A_s}$ for space-time centres near $(0,1)$.

\begin{prop}\label{prop:F.expand2}
Let $X(s)$ denote the position of $L_{A_s}$ given in \eqref{eq:LAs}. Then there exists $s_0>0$ and $r_0>0$ such that whenever  $(x_0,t_0)$ lies in the set
 $$S=\{(x_0,t_0)\in\bC^2\times\bR^+\,:\,|x_0|^2 + 2|t_0-1|^2\leq r_0^2\}, $$  
 and $|s|\leq s_0$ we have
\begin{equation*}
F(X(s),x_0,t_0)\leq F(X(0),0,1) - \frac{\pi}{4e} (|x_0|^2 + 2|t_0-1|^2) - \frac{2\pi}{9e} s^6. 
\end{equation*}
\end{prop}

\begin{proof}
We know from \eqref{eq:vols} that if $x_0\in\bC^2$ and $t_0\in\bR^+$ then
\begin{equation*}
F(X(s),x_0,t_0)=\frac{1}{2\pi e}\int_0^{2\pi}\int_0^{2\pi}I(s,x_0,t_0)\rd\theta_1\rd\theta_2,
\end{equation*}
where
\begin{equation*}
I(s,x_0,t_0)=\frac{1}{t_0}\sqrt{\cosh^22s-\sinh^22s\cos^2(\theta_1+\theta_2)}
e^{1-\frac{|X(s)-x_0|^2}{4t_0}}.
\end{equation*}
Pick any $(\xi,\tau) \in\bC^2\times\bR$ with $|\xi|^2 + 2|\tau|^2 = 1 $, and define
$$ f(r,s)= \int_0^{2\pi}\int_0^{2\pi}I(s,r \xi, 1 + r\tau)\rd\theta_1\rd\theta_2.$$
Performing a Taylor expansion using \eqref{eq:Xs} (and with the help of Mathematica) around $(r,s)=(0,0)$ yields
\begin{equation*}\begin{split}
 f(r,s) &= f(0,0)  - \pi^2 (|\xi|^2 + 2 \tau^2) r^2  - \frac{8}{9}\pi^2 s^6 + O(r^2s) + O(r^3) + O(s^7)\\
&= f(0,0)  - \pi^2 r^2  - \frac{8}{9}\pi^2 s^6 + O(r^2s) + O(r^3) + O(s^7).
\end{split}
\end{equation*}
We can thus choose $r_0,s_0>0$ sufficiently small  such that for $ |r| \leq r_0$ and $|s| \leq s_0$ we have
$$ f(r,s) \leq f(0,0)  - \frac{1}{2} \pi^2 r^2  - \frac{4}{9}\pi^2 s^6.$$
Since this estimate is uniform in $(\xi,\tau)$, this yields the desired statement.
\end{proof}

We can now combine Propositions \ref{prop:F.expand} and \ref{prop:F.expand2} to give our first key result.

\begin{thm}\label{thm:entropy.expand}
For $s$ near $0$ we have that
\begin{equation*}
\lambda(X(s))\leq \lambda(X(0))-\frac{2\pi}{9e}s^6.
\end{equation*}
Hence, the entropy $\lambda(X(s))$ has a local maximum at $s=0$.
\end{thm}

\begin{proof}
 We first recall that Huisken's monotonicity formula \cite{HuiskenMonotonicity} implies that for a compact self-shrinker $\Sigma$ satisfying \eqref{eq:shrinker}, the entropy $\lambda(\Sigma)$ is uniquely attained at $(0,1)$: we consider the self-similar evolution of $\Sigma$ given by
$\Sigma_{t} = \sqrt{-t} \cdot \Sigma$ for $t \in (-\infty, 0)$. The first observation is that the monotonicity formula implies that the Gaussian density at minus infinity satisfies
\begin{equation*}
\Theta((\Sigma_t)_{t<0}, \infty) = \lambda(\Sigma) .
\end{equation*}
Since this flow is self-similar, this yields that 
\begin{equation*}
 \lambda(\Sigma) = F(\Sigma, 0,1).
 \end{equation*}
Now assume that there is a point $(x_0,t_0) \neq (0,1)$ such that $ \lambda(\Sigma) = F(\Sigma, x_0,t_0)$. The monotonicity formula then implies that $(\Sigma_t)_{t<0}$ is also self-similarly shrinking with respect to the point $(x_0, t_0 -1)$. This already yields that $t_0=1$. The monotonicity formula further implies that the entropy is attained on any point along the line containing $x_0$ and $0$, and thus $\Sigma$ has to split as a product $\Sigma' \times \mathbb{R}$. This contradicts the compactness of $\Sigma$.

Now consider $\Sigma'$ given as an exponential normal graph of $U\in C^{\infty}(N\Sigma)$. We choose $\varepsilon_0>0$ and assume 
\begin{equation}\label{eq:U.epsilon.bound}
\|U\|_{C^1} \leq \varepsilon\leq \varepsilon_0. 
\end{equation}
Note that this implies that for $\varepsilon_0 = \varepsilon_0(\Sigma)>0$ sufficiently small, given any $\eta_0>0$, there exists a $\delta_0 = \delta_0(\Sigma, \eta_0) >0$ such that 
\begin{equation}\label{eq:gaussian estimate}
 F(\Sigma', x_0,t_0) \leq 1+\eta_0
 \end{equation}
for all $x_0 \in \mathbb{C}^2$,  $0<t_0<\delta_0$. We can choose $\eta_0 = \frac{1}{4}(\lambda(\Sigma) -1)>0$ (as $\Sigma$ is not a plane). Since the entropy of $\Sigma$ is uniquely attained at $(0,1)$, given any $r>0$, there exists $0<\eta<\eta_0$ such that
$$ F(\Sigma,x_0,t_0) < \lambda(\Sigma) - 3\eta$$
for all $|x_0| > r$ and $(t_0-1)^2 > r$. Using \eqref{eq:gaussian estimate} we see that we can thus choose $\varepsilon$ sufficiently small in \eqref{eq:U.epsilon.bound} such that
\begin{equation*}
F(\Sigma',x_0,t_0) < \lambda(\Sigma) - 2\eta
\end{equation*}
for all $|x_0| \geq r$ and $(t_0-1)^2 \geq r$ and 
\begin{equation*}
 F(\Sigma',0,1) \geq \lambda(\Sigma)  - \eta .
 \end{equation*}
We deduce that the entropy of $\Sigma'$ is attained in the set $$\{ (x_0,t_0)\in\bC^2\times\bR^+\,:\, |x_0| \leq r, (t_0-1)^2 \leq r\}.$$
Applying this to our set-up, we see that for $s$ small, the entropy $\lambda(X(s))$ is only attained at (possibly non-unique) points $(x_s,t_s)$ with the property $(x_s,t_s) \rightarrow (0,1)$ as $s \rightarrow 0$. The claimed result then follows directly from Proposition \ref{prop:F.expand2}.
\end{proof}

Theorem \ref{thm:entropy.expand} yields the following immediate corollary.

\begin{cor}\label{cor:local.min.entropy}
The Clifford torus is not a local entropy minimiser, even under Hamiltonian variations.
\end{cor}

Given that the Clifford torus is the simplest example of a compact Lagrangian self-shrinker in $\bC^2$, Corollary \ref{cor:local.min.entropy} naturally leads one to ask: which Lagrangian self-shrinkers in $\bC^2$ are local minimisers of the entropy under Hamiltonian variations?

\vspace{-4pt}
\subsection{Flow instability}  With these results in hand, we can now prove our flow instability result.

\begin{thm}\label{thm:flow.instability}
For every $\epsilon>0$ and $k\in\bN$, there exists a compact embedded Lagrangian torus $L'$, Hamiltonian isotopic and $\epsilon$-close in $C^{k}$ to the Clifford torus $L$, such that Lagrangian mean curvature flow starting at $L'$ develops a first finite-time singularity  whose blow-up is not $A\cdot L$ for any $A\in \UU(2)$.  Hence, the rescaled Lagrangian mean curvature flow starting at $L'$ does not converge to the Clifford torus.
\end{thm}

\begin{proof}
We can choose $L'=L_{A_s}$ for some $s$ sufficiently small to ensure it is $\epsilon$-close in $C^{k}$ to $L$.  The entropy $\lambda(X(s))$ is strictly less than the entropy of $L$ and the entropy is non-increasing under the Lagrangian mean curvature flow (or the rescaled flow) by Lemma \ref{lem:entropy}.  Hence, the rescaled flow cannot converge to any member of the $\UU(2)$-orbit of $L$.
\end{proof}

Theorem \ref{thm:flow.instability} yields the following interesting corollary, which is surprising given the local stability results in Lemma \ref{lem:L.HamFstable} and Theorem \ref{thm:L.localmin.vol}.  In particular, even though the Clifford torus is locally volume minimizing under Hamiltonian variations, it is unstable under Lagrangian mean curvature flow under such perturbations.  This is counter to one's intuition concerning gradient flows, and points to a lack of the expected coercivity of the volume functional in this situation.

\begin{cor}
The Clifford torus is Hamiltonian unstable for arbitrarily small deformations under Lagrangian mean curvature flow.
\end{cor}

 It is worth noting that the work in \cite{NevesFTS} implies that  we can find $L'$ Hamiltonian isotopic to $L$ as in Theorem \ref{thm:flow.instability} which is arbitrarily $C^0$-close, but not $C^1$-close.  The fact that $L'$ given in \cite{NevesFTS} is not $C^1$-close to $L$ is not a technicality, but rather an essential consequence of the construction of $L'$.  Hence, Theorem \ref{thm:flow.instability} strengthens the results of \cite{NevesFTS} in this instance.

The family $L_{A_s}$ we chose in the proof of Theorem \ref{thm:flow.instability} is, up to unitary transformation, the same family considered in \cite[$\S$4]{Grohetal} and \cite[Theorem C]{NevesMonotone}.  The latter implies the following.

\begin{thm}\label{thm:Neves.LAs} For $s$ sufficiently large, Lagrangian mean curvature flow starting at $L_{A_s}$ will develop a first finite-time Type II singularity at the origin, whose blow-up is a transverse pair of special Lagrangian planes with the same Lagrangian angle.
\end{thm}

This shows Hamiltonian instability of the Clifford torus for large deformations (i.e.~large $s$), but we have now shown it is true for any sufficiently small deformation (so small $s$).  It is reasonable to ask whether the same behaviour as in Theorem \ref{thm:Neves.LAs} occurs for Lagrangian mean curvature flow starting at $L_{A_s}$ for any $s\neq 0$.  Since $L_{A_s}$ is monotone, by the theory in \cite{NevesMonotone}, it should be enough to show that the first singularity of the flow  starting at $L_{A_s}$ is before time $\cosh 2s$.  This is equivalent to saying that the flow starting at $\frac{1}{\sqrt{\cosh 2s}}L_{A_s}$ becomes singular before time $1$, which is when the Clifford torus (and any self-shrinker satisfying \eqref{eq:shrinker}) shrinks to a point.

\vspace{-4pt}
\subsection{Stability}

In contrast to our instability results we can prove a stability result for the Clifford torus as follows, which utilises our local uniqueness result we shall prove later.

\begin{thm}\label{thm:flow.stability}
Let $L'$ be a compact embedded Lagrangian in $\mathcal{S}^3(2)$,  Hamiltonian isotopic to the Clifford torus $L$.  If $L'$ is sufficiently close to $L$, then $L'$ has a first finite-time Type I singularity at the origin at time 1 and the rescaled Lagrangian mean curvature flow starting at $L'$ converges to $L$, up to some unitary transformation.
\end{thm}

\begin{proof}
For any Lagrangian $L'$ in $\mathcal{S}^3(2)$ with position vector $X'$, we see that 
\begin{equation*}
4\pi F(X',0,1)=\Vol(L').
\end{equation*} By Theorem \ref{thm:L.localmin.vol}, we therefore know that for any $L'$ as in the statement we have that
\begin{equation}\label{eq:stab.1}
4\pi F(X',0,1)=\Vol(L')\geq \Vol(L)=4\pi F(X,0,1).
\end{equation} 
We can also deduce this result directly from our own calculations.  

By the work in \cite{CasLerMiq}, we know that $L'$ is a torus foliated by Hopf circles, so we can write
\begin{equation*}
L'=\{e^{i(\frac{\theta_1+\theta_2}{2})}\gamma(\theta_1-\theta_2)\,:\,\theta_1,\theta_2\in\bR\}
\end{equation*}
for some curve $\gamma(\theta_1-\theta_2)$ in $\mathcal{S}^2(2)$.  In the case of $L$, we have that
\begin{equation*}
\gamma(\theta_1-\theta_2)=\sqrt{2}(e^{i(\frac{\theta_1-\theta_2}{2})},e^{-i(\frac{\theta_1-\theta_2}{2})}).
\end{equation*}
Therefore, the variation vector field of any Hamiltonian isotopy at $L$ is given by $J\nabla f$ where $f=f(\theta_1-\theta_2)$. 
The Hamiltonian vector fields $V$ on $L$ for which $\langle V,(\Delta_L^{\perp}-1)V\rangle\leq 0$ are given in Lemmas \ref{lem:translations} and \ref{lem:f.1.espace2}.  We see immediately that the only such $V$ of the form $J\nabla f$ where $f$ is a function of $\theta_1-\theta_2$ are $Y_1^{\perp}$ and $Y_2^{\perp}$ which generate unitary transformations.  Thus the second variation of $F$ at $(X,0,1)$ is non-negative for Hamiltonian variations within $\mathcal{S}^3(2)$ and the directions for which the second variation vanishes are given by unitary transformations, under which the $F$-functional is constant.
We therefore have that \eqref{eq:stab.1} holds as claimed.

We now consider the maximal smooth evolution $(L'_t)_{0\leq t<T}$ of $L'=:L'_0$ by (Lagrangian) mean curvature flow. We first recall that also for mean curvature flow of higher codimension, more precisely for mean curvature flow of $k$-dimensional surfaces in $\mathbb{R}^n$, spheres with radius $R(t) = \sqrt{R^2-2kt}$ act as barriers both from the inside and from the outside. Applied to the present set-up this implies that 
\begin{equation*}
 L'_t \subseteq \mathcal{S}^3(2 \sqrt{1-t}) 
  \end{equation*}
for all $0\leq t < T$. Note further that since Lagrangian mean curvature flow preserves the Maslov class $[\rd\theta']$ of $L'$ and the class of the Liouville form of $L'_t$ satisfies 
\begin{equation*}
[\lambda'_t]=2(1-t)[\rd\theta']
\end{equation*} by \cite[Lemma 2.1]{NevesMonotone}, we have that $L'_t$ remains Hamiltonian isotopic to $L_t$ where $(L_t)_{0\leq t<1}$ is the self-similar evolution of the Clifford torus.

We consider the rescaled flow
\begin{equation*}
 \tilde{L}'_\tau = \frac{1}{\sqrt{1-t}}L'_t\subseteq \mathcal{S}^3(2),
  \end{equation*}
where $\tau = - \log(1-t)$, and let $\tilde{X}'_{\tau}$ denotes its position vector in $\bC^2$. Note that this rescaling yields $\tilde{L}_\tau = L$. We now consider Huisken's rescaled monotone quantity:
\begin{equation*}
\mathcal{E}(\tilde{L}'_\tau) = \frac{1}{4\pi}\int_{\tilde{L}'_\tau}  \exp\left(-|\tilde{X}'_{\tau}|^2/4\right)\vol_{\tilde{L}'_{\tau}}=\frac{1}{4\pi e}\Vol(\tilde{L}'_{\tau}),
 \end{equation*}
 which is decreasing in $\tau$. Furthermore, by \eqref{eq:stab.1}, we have
\begin{equation}\label{eq:stab.2}
 \mathcal{E}(\tilde{L}'_\tau) \geq  \mathcal{E}(L)  = \frac{2\pi}{e}.
 \end{equation}
The remaining argument is now a direct application of the \L ojasiewicz-Simon inequality as in \cite{Schulze14}, which we now outline.

We assume that $L'$ can be written as a normal exponential graph over $L$, given by $V\in C^{\infty}(NL)$. Then we can write, at least for $\tau$ sufficiently small, $\tilde{L}'_\tau$ as normal exponential graphs over $L$, given by $V(\tau)\in C^{\infty}(NL)$. 

Let $0<\varepsilon<\sigma_0/2$ be chosen later, and assume that
\begin{equation*}
\|V(0)\|_{C^{2,\alpha}}<\varepsilon.
\end{equation*}
 We consider the set 
\begin{equation*}
 S := \{\tau>0\, :\, \|V(s)\|_{C^{2,\alpha}}\leq\sigma_0 \text{ for all } s \in [0,\tau)\}.
  \end{equation*}
  We aim to show that for $\varepsilon$ sufficiently small and $\sigma_0$ chosen suitably $S = [0,\infty)$.
  
Note that for $\sigma_0$ sufficiently small, since $\varepsilon<\sigma_0/2$, there exists a $\delta >0$, independent of $\varepsilon$ such that $(0,2\delta]\subseteq S$. By higher interior estimates, see for example \cite{White05}, there exists $C_0>0$ such that for all $\varepsilon\in(0,\sigma_0/2)$ we have
\begin{equation}\label{eq:stab.3}
 \|V(\tau)\|_{C^{3,\alpha}} \leq C_0
 \end{equation}
for all $\tau \in S$ with $\tau \geq \delta$.  

We write $\mathcal{E}(V(\tau)) =  \mathcal{E}(\tilde{L}'_\tau)$ and $\mathcal{E}(0) =  \mathcal{E}(L)$. For $\sigma_0$ sufficiently small, depending only on $L$, we have by \cite[Lemma 3.1]{Schulze14}  that there exists $\theta \in (0,1/2)$ and $C_1 = C_1(L)$ such that
\begin{equation}\label{eq:energy.theta.bound}
\sup_{\tau \in [\tau_1, \tau_2]} \|V(\tau)-V(\tau_1)\|_{L^2} \leq C_1 \left(\mathcal{E}(V(\tau_1)) - \mathcal{E}(0)\right)^\theta
\end{equation}
for all $\tau_1,\tau_2\in S$, $\tau_1<\tau_2$. We now fix $\sigma_0$ accordingly so that \eqref{eq:energy.theta.bound} holds. Note that $\mathcal{E}(V(0)) \rightarrow \mathcal{E}(0)$ as $\varepsilon \rightarrow 0$. This implies that 
\begin{equation*}
 \|V(\tau)\|_{L^2} \leq \|V(0)\|_{L^2} +  C_1 \left(\mathcal{E}(V(0)) - \mathcal{E}(0)\right)^\theta .
  \end{equation*}
for all $\tau \in S$. Interpolating the $C^{2,\alpha}$-norm between the $L^2$-norm and the $C^{3,\alpha}$-norm and using \eqref{eq:stab.3}, we see that for $\varepsilon$ sufficiently small we have
\begin{equation*}
 \|V(\tau)\|_{C^{2,\alpha}}\leq \frac{3}{4}\sigma_0
  \end{equation*}
for all $\tau \in S$ with $\tau \geq \delta$, and thus $S = [0,\infty)$ as desired. 

The monotonicity formula then implies that there is a sequence $\tau_i\rightarrow \infty$ such that $\tilde{L}'_{\tau_i}$ converges smoothly to a self-shrinker $L''$ which is a small $C^{2,\alpha}$ normal graph of, say, $V''$ over $L$.

By the \L ojasievicz--Simon inequality we have that $\mathcal{E}(L'') = \mathcal{E}(L)$ and thus by the monotonicity of $\mathcal{E}(V(\tau))$,
\begin{equation*}
 \mathcal{E}(V(\tau)) \rightarrow \mathcal{E}(0)
  \end{equation*}
as $\tau \rightarrow \infty$. Thus by \eqref{eq:energy.theta.bound} we have that $V(\tau)$ is a Cauchy sequence in $L^2(NL)$ and the sequence converges to $V''$.

 By the local uniqueness of the Clifford torus, Theorem \ref{thm:local.uniq}, we have that $L'' = A\cdot L$ for some $A\in\UU(2)$ and the whole sequence converges.  As an alternative to using Theorem \ref{thm:local.uniq} below, we may observe that the argument thus far implies that $L''$ must be a minimal Lagrangian torus contained in $\mathcal{S}^3(2)$, which is embedded as it is a small $C^{2,\alpha}$ graph over $L$.  Hence, $L''$ is $L$ up to a unitary transformation by the proof of the Lawson Conjecture \cite{Lawson}.
\end{proof}

We should be clear that Theorem \ref{thm:flow.stability} actually holds without the additional assumption that $L'$ is close to $L$, as stated in Theorem \ref{thm:Hamiso.CLM} below.  This follows from the work in \cite{CasLerMiq}, as we shall now explain.  However, we wanted to illustrate here in the proof of Theorem \ref{thm:flow.stability} an alternative approach for obtaining an (albeit weaker) stability result for the Clifford torus, which may be applicable in other contexts where the special techniques implemented in \cite{CasLerMiq} may not be valid.

To relate Theorem \ref{thm:flow.stability} to work in \cite{CasLerMiq} we require the following result in curve shortening flow  (cf.~\cite{CasLerMiq,Gage}), which is interesting in its own right. 

\begin{thm}\label{thm:curve.short.s2}
Let $\gamma_0$ be a simple closed curve in $\mathcal{S}^2$ and let $\gamma_t$ be the evolution of $\gamma_0$ under curve shortening flow in $\mathcal{S}^2$.  Then the following are equivalent.
\begin{itemize}
\item[(a)] $\gamma_0$ is Hamiltonian isotopic to an equator in $\mathcal{S}^2$.\\[-2ex]
\item[(b)] $\gamma_0$ divides $\mathcal{S}^2$ into two regions of equal area.\\[-2ex]
\item[(c)] $\gamma_t$ divides $\mathcal{S}^2$ into two regions of equal area for all $t$.\\[-2ex]
\item[(d)] $\gamma_t$ exists for all time and converges to an equator in $\mathcal{S}^2$.
\end{itemize}
\end{thm}

\begin{proof}
By definition, Hamiltonian isotopies in $\mathcal{S}^2$ preserve area (as the symplectic form on $\mathcal{S}^2$ is the area form), so if $\gamma_0$ is Hamiltonian isotopic to an equator, it divides $\mathcal{S}^2$ into two regions of equal area.  Therefore, (a) implies (b).

By \cite[Lemma 3.2]{CasLerMiq}, $\gamma_0$ divides $\mathcal{S}^2$ into two regions of equal area if and only if $\gamma_t$ does for each $t$.  Therefore, (b) is equivalent to (c).   Moreover, the proof of \cite[Corollary 3.3]{CasLerMiq} states that (b) is equivalent to (d).
Therefore, we need only show that $\gamma_t$ is a Hamiltonian isotopy to show that (d) implies (a) and thus complete the proof.

At every time $t$, we know that $\gamma_t$ is a simple closed curve and divides $\mathcal{S}^2$ into two regions of equal area $2\pi$ (using here that $\mathcal{S}^2$ has curvature $1$ and so its area is $4\pi$).  Choose one of these regions $\mathcal{U}_t$ for each $t$.  By  Gauss--Bonnet,
\begin{equation*}\int_{\mathcal{U}_t} \, \rd A + \int_{\gamma_t} \kappa_t = 2\pi \chi(\mathcal{U}_t),
\end{equation*}
where $\kappa_t$ is the curvature of $\gamma_t$. Since the area of $\mathcal{U}_t$ is $2\pi$ and the Euler characteristic $\chi(\mathcal{U}_t)=1$ we have that
\begin{equation*}
\int_{\gamma_t}\kappa_t=0
\end{equation*}
for all $t$.  Hence, $\kappa_t$ is exact for each $t$, and thus the curve shortening flow 
$\gamma_t$ is indeed a Hamiltonian isotopy as desired.
\end{proof}

We now make the following observation.

\begin{lem}\label{lem:vol.split}
A compact embedded orientable Lagrangian $L'$ in $\mathcal{S}^3(2)$ is Hamiltonian isotopic to the Clifford torus $L$ if and only if $L'$ divides $\mathcal{S}^3(2)$ into two regions of equal volume. 
\end{lem}

\begin{proof}
	In \cite[Proposition 2.1]{CasLerMiq}, it is shown that any embedded Lagrangian torus in $\mathcal S^3$ descends via the Hopf fibration $\pi: \mathcal S^3 \to \mathcal S^2$ to a simple closed curve on $\mathcal S^2$. We review and extend this argument to show that we can translate the stated claim to one involving curves on $\mathcal{S}^2$. 

Let $L' = L_0$ and $L = L_1$ be Hamiltonian isotopic through Lagrangian tori $L_s \subset \mathcal{S}^3(2)$, $s\in [0,1]$, and let 
$\tilde{f}_s:L_s\to\bR$ be smooth Hamiltonian functions generating the isotopy.  Let $N$ denote the normal vector field to $\mathcal{S}^3(2) \subset \bC^2$. Since $L_s$ is Lagrangian, $JN$ is tangent to $L_s$. The integral curves of $JN$ are the Hopf circles, so $L_s$ must be foliated by such circles, and hence $\pi(L_s)=\gamma_s\subseteq\mathcal{S}^2$ is a closed curve. We also see that $L_s$ is embedded if and only if $\gamma_s$ is simple.  Furthermore, as $J \nabla \tilde f_s$ is tangent to $\mathcal S^3(2)$, we have that
\begin{equation*}
\langle \nabla \tilde f_s, JN\rangle = -\langle J\nabla\tilde f_s,N\rangle = 0
\end{equation*}
and so $\tilde f_s$ is constant along the Hopf fibres. Thus, the $\tilde f_s$ descend to Hamiltonian functions $f_s : \gamma_s \to \bR$.  Conversely, given Hamiltonian  functions $f_s$ generating an isotopy  $\gamma_s$, we can lift $\gamma_s$ to a Hamiltonian isotopy $L_s$ by extending each $f_s$ to a function $\tilde f_s$ constant along each Hopf fibre. In conclusion, we have a one-to-one correspondence between Hamiltonian isotopies of Lagrangian tori in $\mathcal S^3(2)$ and of closed curves in $\mathcal S^2$.

Moreover, one may easily see, as in \cite{CasLerMiq}, that given any embedded Lagrangian torus $L'$ in $\mathcal{S}^3(2)$, the ratio of the volumes of the two regions of $\mathcal{S}^3(2)$ determined by $L'$ is equal to the ratio of the areas of the two regions of $\mathcal{S}^2$ determined by the simple closed curve $\pi(L')$.  

The result then follows from Theorem \ref{thm:curve.short.s2}.
\end{proof}

With this result in hand, we can re-cast the main results of \cite{CasLerMiq} as follows, which thus shows that Theorem \ref{thm:flow.stability} is a special case of their work.

\begin{thm}\label{thm:Hamiso.CLM}
Let $L'$ be a compact embedded orientable Lagrangian in $\mathcal{S}^3(2)$.
\begin{itemize}
\item[(a)] If $L'$ is Hamiltonian isotopic to the Clifford torus $L$, then Lagrangian mean curvature flow starting at $L'$ has a first finite-time Type I singularity at the origin, and the rescaled flow converges to $L$, up to a unitary transformation.\\[-2ex]
\item[(b)] If $L'$ is not Hamiltonian isotopic to $L$, then Lagrangian mean curvature flow starting at $L'$ has a first finite-time Type I singularity along a circle, and the rescaled flow converges to a cylinder $\mathcal{S}^1\times\bR$ in some $\bR^3\subseteq\bC^2$.
\end{itemize}
\end{thm}

We should note that the rescaling in (b) considered in \cite{CasLerMiq} is not the standard rescaling, but it is equivalent to the standard one and so the result holds as stated.  It is perhaps interesting to observe that 
the entropy of the cylinder is
\begin{equation*}
\lambda(\mathcal{S}^1\times\bR)=\sqrt{\frac{2\pi}{e}}=1.520\ldots
\end{equation*}
which is less than $2$ (the entropy of two planes) and less than the entropy of the Clifford torus.

We also observe the following corollary, which is also known by Lemma \ref{lem:vol.split} and the study of the isoperimetric problem in  
$3$-dimensional space forms in \cite{RitoreRos}, for example.

\begin{cor}
The Clifford torus $L$ is the unique volume (and thus entropy) minimiser amongst Lagrangians in $\mathcal{S}^3(2)$ Hamiltonian isotopic to $L$, up to unitary transformations.
\end{cor}

\begin{proof}
By Theorem \ref{thm:Hamiso.CLM} any Lagrangian $L'$ Hamiltonian isotopic to $L$ will under the rescaled mean curvature flow converge to $L$, up to some unitary transformation.  Since the entropy is non-increasing along the flow we know that $\lambda(X')\geq\lambda(X)$ (and hence $\Vol(L')\geq \Vol(L)$).  We also know that the entropy is constant if and only if the flow is self-similar, but then $L'$ must be $L$ up to a unitary transformation.
\end{proof}

\section{Local uniqueness}\label{sec:rigidity}

We now wish to move away from the purely Lagrangian setting and discuss the local uniqueness of the Clifford torus as a self-shrinker for mean curvature flow.  This is not straightforward since there is a kernel for the linearisation for the self-shrinker equation which is larger than we would expect: i.e.~it does not just consist of infinitesimal rigid motions.  Therefore, we must show the remaining infinitesimal deformations are genuinely obstructed  to deduce local uniqueness.

\vspace{-4pt}
\subsection{The self-shrinker equation}

We start by observing that any compact embedded submanifold which is a graph over the Clifford torus can be written as the image of an immersion
\begin{equation*}
X_V=X+V:\mathcal{S}^1\times\mathcal{S}^1\to\bC^2
\end{equation*}
where $V$ is a normal vector field on the Clifford torus $L$, which is the image of $X$.  Moreover, the graph of $V$ must lie in a $C^1$-neighbourhood $\mathcal{U}$ of the zero section in the normal bundle $NL$ (where we omit the inclusion of the pullback of this bundle to $\mathcal{S}^1\times\mathcal{S}^1$ for simplicity).  We therefore denote the image of $X_V$ by $L_V$ for $V\in C^1(\mathcal{U})$, where the notation means $C^1$ sections of $NL$ whose graph lies in $\mathcal{U}$, and we use similar notation for sections of other Banach spaces.  We also let $\mathcal{T}$ denote the tubular neighbourhood of $L$ given by applying the exponential map to $\mathcal{U}$.

We know that $L_V$ is a self-shrinker with space-time centre at $(x_0,t_0) = (0,1) \in \bC^2\times\bR^+$  if and only if \eqref{eq:shrinker} is satisfied.  We can equivalently say that $C^{1}$-close self-shrinkers $L_V$ are characterised as zeros of the functional 
\begin{equation}\label{eq:S}
\begin{split}
&\mathcal{S}:C^{2,\alpha}(\mathcal{U})\to C^{0,\alpha}(NL)\\
&\mathcal{S}(V)= -X_V^*\left(H(X+V)+\frac{(X+V)^{\perp_V}}{2}\right)^{\perp}.
\end{split}
\end{equation}
Here, $H(X+V)$ is the mean curvature vector of $L_V$ and ${}^{\perp_V}$, ${}^{\perp}$ denote the orthogonal projections on $NL_V$ and $NL$ (again abusing notation and omitting pull backs).  Since $L_V$ is a normal graph over $L$, the projection of the vector in brackets in \eqref{eq:S} onto $NL$ will vanish if and only if the vector vanishes. 
Moreover, we need only consider $V$ in $C^{2,\alpha}$ since if the self-shrinker equation is satisfied then $V$ will necessarily be smooth.

\vspace{-4pt}
\subsection{Rotations}
We know that the action of rotations preserves the condition \eqref{eq:shrinker}.  

To deal with this, recall the normal vector fields $Y_j^{\perp}$ on $L$ for $j=1,2,3,4$ given in \eqref{eq:Y1perp}, \eqref{eq:Y2perp} and \eqref{eq:Y3perp.Y4perp}.  Define
\begin{equation}\label{eq:piY}
\pi_{\mathcal{Y}}:C^{2,\alpha}(NL)\to \mathcal{Y}=\Span\{Y_1^{\perp},Y_2^{\perp},Y_3^{\perp},Y_4^{\perp}\}
\end{equation}
to be $L^2$-orthogonal projection.

\begin{lem}\label{lem:slice}
Making $\mathcal{U}$ smaller if necessary, for any sufficiently $C^1$-close submanifold $L'$ to $L$ there exists $A\in \SO(4)$ and $V\in C^1(\mathcal{U})\cap \Ker\pi_{\mathcal{Y}}$ such that $A\cdot L'=L_{V}$.  Moreover, $A$ is unique up to the action of $\UU(1)^2$ preserving $L$.
\end{lem}

\begin{proof}
This is a direct application of the slice theorem for Lie group actions by diffeomorphisms.
\end{proof}

This yields a description of self-shrinkers satisfying \eqref{eq:shrinker} which are close to $L$, modulo the action of rotations.

\begin{lem}\label{lem:shrinker}
Up to the action of rotations, sufficiently $C^1$-close self-shrinkers to $L$ are uniquely determined by zeros of the functional
\begin{equation}\label{eq:S0}
\begin{split}
&\mathcal{S}_0:C^{2,\alpha}(\mathcal{U})\cap\Ker\pi_{\mathcal{Y}}\to C^{0,\alpha}(NL)\\
&\mathcal{S}_0(V)=-X_V^*\left(H(X+V)+\frac{(X+V)^{\perp_V}}{2}\right)^{\perp},
\end{split}
\end{equation}
where $\pi_{\mathcal{Y}}$ is given in \eqref{eq:piY}.  The  linearisation of $\mathcal{S}_0$ at $0$ is given by
\begin{equation}\label{eq:S0.linear}
\mathcal{L}_{0}=\Delta_L^{\perp}-1:\Ker\pi_{\mathcal{Y}}\to C^{0,\alpha}(NL).
\end{equation}
Hence,
\begin{equation*}
\mathcal{S}_0(V)=(\Delta_L^{\perp}-1)V+\mathcal{Q}_0(V),
\end{equation*}
for some smooth functional $\mathcal{Q}_0$ whose value and first derivatives at $0$ vanish.  Moreover, the kernel of $\mathcal{L}_0$ is given by
\begin{equation*}
\mathcal{K}=\Ker\mathcal{L}_{0}=\Span\{V_A^{\perp},V_B^{\perp},V_C^{\perp}, V_D^{\perp}\},
\end{equation*}
using the notation of \eqref{eq:VAperp}, \eqref{eq:VBperp} and \eqref{eq:VCperp.VDperp}.
\end{lem}

\begin{proof}
If we are given a family $X(s)$ with 
\begin{gather*}
X(0)=X,\quad 
\frac{\partial X(s)}{\partial s}=V(s), 
\end{gather*} then the first variation formula for $F=F(X(s),0,1)$ for any $s$  is (\cite[Theorem 1]{LeeLue}):
\begin{equation}\label{eq:first.var}
\begin{split}
\frac{\partial F}{\partial s}&=-\frac{1}{4\pi}\int_{L(s)}
\left\langle V(s), H(s)+\frac{X(s)}{2}\right\rangle e^{-\frac{|X(s)|^2}{4}}\vol_{L(s)},
\end{split}
\end{equation}
where $L(s)$ is the image of $X(s)$, and the rest of the notation should be clear.
Differentiating \eqref{eq:first.var} with respect to $s$ and setting $s=0$, we can compare the result to the second variation in Lemma \ref{lem:L.secondvar} and deduce that
\begin{align*}
-\Big\langle V(0),\frac{\partial}{\partial s}\Big(H(s)+\frac{X(s)}{2}&\Big)\Big|_{s=0}\Big\rangle_{L^2}
=
\left\langle V(0),(\Delta_L^{\perp}-1)V(0)
\right\rangle_{L^2}.
\end{align*}
Thus the linearisation of $\mathcal{S}_0$ at $0$ is as given in \eqref{eq:S0.linear} and the expression for $\mathcal{S}_0(V)$ is as claimed.

The remainder of the result follows from the discussion at the start of this subsection, Lemmas \ref{lem:V.1.espace.2} and \ref{lem:slice}.
\end{proof}

For convenience later, we let 
\begin{equation*}
\pi_{\mathcal{K}}:C^{2,\alpha}(NL)\to \mathcal{K}=\Span\{V_A^{\perp},V_B^{\perp},V_C^{\perp}, V_D^{\perp}\}
\end{equation*}
denote $L^2$-orthogonal projection onto $\mathcal{K}$.

\subsection{Deformations and obstructions}\label{sec:def and ob}

In this subsection, we explain the general strategy for studying the problem of deformations and obstructions for geometric objects given by the zeros of a function. To illustrate the strategy, we describe how to determine if an element in the kernel of the linearization gives rise to a 1-parameter family of solutions, i.e.~the question of integrability.

Let $\mathcal{X}$ and $\mathcal{Y}$ be (Banach) spaces of sections of vector bundles and let $\mathcal{F}:\mathcal{X}\to\mathcal{Y}$ be a smooth map.  In our particular context, $\mathcal{X}=\mathcal{Y}=C^{2,\alpha}(NL)$ and $\mathcal{F}=\mathcal{S}_0$.  We are interested in studying $\mathcal{F}^{-1}(0)$.

We suppose that, for $V$ sufficiently near $0$ in $\mathcal{X}$, we have that
$$\mathcal{F}(V)=\sum_{n=0}^{\infty}\frac{\mathcal{F}_n(V,\ldots,V)}{n!}$$
where $\mathcal{F}_n:S^n\mathcal{X}\to\mathcal{Y}$ is the ``$n$th derivative of $\mathcal{F}$ at $0$'', which is literally true for $\mathcal{F}_1$.  We will assume now that $\mathcal{F}(0)=0$ so that $\mathcal{F}_0=0$.  This is true in our setting, since $\mathcal{F}(0)=0$ is the statement that $L$ is a self-shrinker, and we have that $\mathcal{F}_1=\mathcal{L}_0$.

To demonstrate the general strategy, we suppose that we have 1-parameter family $V(s)$ in $\mathcal{X}$ so that $\mathcal{F}(V(s))=0$ for all $s$ sufficiently small, where $V(s)$ is analytic in $s$. For the question of rigidity, we later will have to adapt the argument, assuming that we only have a sequence $V_k \rightarrow 0$ such that $\mathcal{F}(V_k)=0$, where $V_k$ cannot necessarily be written as a sequence $V(s_k)$ for an analytic function $V(s)$. 
 
Since we assume that $V(s)$ is analytic in $s$ we can write
$$V(s)=\sum_{n=0}^{\infty} \frac{V_n}{n!}s^n$$
where $V_n\in\mathcal{X}$ is the ``$n$th derivative of $V(s)$ at $s=0$''.  

We now compute, expanding in $s$,
\begin{align}
\mathcal{F}(V(s))&=s\mathcal{F}_1(V_1)+\frac{s^2}{2}\big(\mathcal{F}_1(V_2)+\mathcal{F}_2(V_1,V_1)\big)+\frac{s^3}{3!}\big(\mathcal{F}_1(V_3)+3\mathcal{F}_2(V_1,V_2)+\mathcal{F}_3(V_1,V_1,V_1)\big)\nonumber\\
&+\ldots+\frac{s^n}{n!}\big(\mathcal{F}_1(V_n)+n\mathcal{F}_2(V_1,V_{n-1})+\ldots+\mathcal{F}_n(V_1,\ldots,V_1)\big)+\ldots\label{eq:F.expansion}
\end{align}
We see immediately from the order $s$ term in \eqref{eq:F.expansion} that 
$$V_1\in\Ker\mathcal{F}_1=\mathcal{K}.$$
We assume that $V_1\neq 0$ (since this is the infinitesimal deformation we wish to extend), and assume that  $\mathcal{K}$ is closed and admits a complement $\mathcal{K}^{\perp}$ (for example, by using an inner product on $\mathcal{X}$).  This notation matches our earlier one above where $\mathcal{K}=\Ker\mathcal{L}_0$ and $\mathcal{K}^{\perp}$ is the $L^2$-orthogonal complement.  This means we can take $V_1=V_A^{\perp}$ without loss of generality, so that
$$V(s)=sV_A^{\perp}+O(s^2).$$

We see that the $n$th term in \eqref{eq:F.expansion} vanishes if and only if
$$\mathcal{F}_1(V_n)=\mathcal{E}_n(V_1,\ldots,V_{n-1}),$$
for some suitable functional $\mathcal{E}_n$ determined by $\mathcal{F}_2,\ldots,\mathcal{F}_n$.  

Therefore, we see that if $\mathcal{F}_1$ is surjective then we can inductively find $V_n$ so that the $n$th term in \eqref{eq:F.expansion} vanishes.  From here, one can recover the well-known result from the Implicit Function Theorem that if $\mathcal{F}_1$ (the linearisation of $\mathcal{F}$ at $0$) is surjective, then for all $V_1\in\mathcal{K}$ there is a curve $V(s)$, for $s$ sufficiently small, so that $V'(0)=V_1$ and $\mathcal{F}(V(s))\equiv 0$.  In other words, every infinitesimal deformation (given by elements of $\mathcal{K}$) can be integrated to path of genuine deformations.

We want to understand when elements $V_1\in\mathcal{K}$ \emph{do not} lead to a genuine deformation, i.e.~$V'(0)=V_1$ but $\mathcal{F}(V(s))\neq 0$ for all $s\neq 0$.  In this case, there must be a term in the expansion \eqref{eq:F.expansion} that \emph{cannot} vanish for \emph{any} $s\neq 0$.

To study this, we suppose that we can write
$$\mathcal{Y}=\Imm\mathcal{F}_1\oplus \mathcal{C}$$
where $\mathcal{C}$ is a closed subspace of $\mathcal{Y}$.  In our setting, $\mathcal{C}$ is the $L^2$-orthogonal complement of $\Imm\mathcal{F}_1$ and is equal to $\mathcal{K}=\Ker\mathcal{L}_0$, as $\mathcal{L}_0$ is self-adjoint.  For ease of notation we write
$$\mathcal{C}^{\perp}=\Imm\mathcal{F}_1.$$
It is worth noting that, in our setting, even though $\mathcal{K}=\mathcal{C}$, we have that $\mathcal{K}^{\perp}\neq\mathcal{C}^{\perp}$, since $\mathcal{K}^{\perp}$ is the complement of $\mathcal{K}$ in $\mathcal{X}=C^{2,\alpha}(NL)$, whereas $\mathcal{C}^{\perp}$ is the complement of $\mathcal{C}=\mathcal{K}$ in $\mathcal{Y}=C^{0,\alpha}(NL)$.  As a result, it is useful to distinguish $\mathcal{K}^{\perp}$ and $\mathcal{C}^{\perp}$.

Since $\mathcal{F}(V(s))=0$ and $V_1\in\mathcal{K}$ we see that the first non-trivial equation we need to solve is the quadratic term in \eqref{eq:F.expansion}:
\begin{equation}\label{eq:obstruction.1}
\mathcal{F}_1(V_2)+\mathcal{F}_2(V_1,V_1)=0.
\end{equation}
Hence, if $\pi_{\mathcal{C}}:\mathcal{Y}\to\mathcal{C}$ is the projection, we see that if
$$\pi_{\mathcal{C}}\big(\mathcal{F}_2(V_1,V_1)\big)\neq 0$$
then \eqref{eq:obstruction.1} can never be satisfied.  This would mean the problem is obstructed and so $V_1$ does not lead to a deformation (or other nearby solution).  In our setting, this means checking the quadratic term in $\pi_{\mathcal{K}}(\mathcal{S}_0(sV_A^{\perp}))$.  We will find in Proposition \ref{prop:VAperp.obs} that this quadratic term vanishes and hence \eqref{eq:obstruction.1} is satisfied.

If \eqref{eq:obstruction.1} can be solved, then $V_2$ can be written uniquely as 
$V_2=U_2+W_2$ where $U_2\in\mathcal{K}$ and $W_2\in\mathcal{K}^{\perp}$.  Moreover, we have that
$$\mathcal{F}_1:\mathcal{K}^{\perp}\to\mathcal{C}^{\perp}$$
is invertible and 
$$W_2=\mathcal{F}_1^{-1}\big(-\mathcal{F}_2(V_1,V_1)\big).$$

We now must move on to the third term in \eqref{eq:F.expansion} and see that we would require
\begin{equation}\label{eq:obstruction.2}
\mathcal{F}_1(V_3)+3\mathcal{F}_2(V_1,V_2)+\mathcal{F}_3(V_1,V_1,V_1)=0.
\end{equation}
Hence, we need 
\begin{equation}\label{eq:obstruction.3}
\pi_{\mathcal{C}}\big(3\mathcal{F}_2(V_1,V_2)+\mathcal{F}_3(V_1,V_1,V_1)\big)\neq 0
\end{equation}
for there to be an obstruction.  Notice that $V_2=U_2+W_2$ where $U_2\in\mathcal{K}$ is arbitrary, but $W_2$ is uniquely determined by $V_1$.  In our setting we find $\mathcal{F}_3(V_1,V_1,V_1)$ as the cubic term in $\mathcal{S}_0(sV_A^{\perp})$, and discover that 
$$\pi_{\mathcal{C}}\big(\mathcal{F}_3(V_1,V_1,V_1)\big)\neq 0.$$
Therefore, what remains to be studied is 
$$\pi_{\mathcal{C}}\big(3\mathcal{F}_2(V_1,U_2+W_2)\big).$$
In our setting this can be done explicitly (since $W_2$ is uniquely determined by $V_1=V_A^{\perp}$ and $U_2$ is a linear combination of elements of $\mathcal{K}$) and we find that it does not cancel the term given by $\pi_{\mathcal{C}}(\mathcal{F}_3(V_1,V_1,V_1))$, i.e.~\eqref{eq:obstruction.3} holds. Hence \eqref{eq:obstruction.2} \emph{cannot} be satisfied in our setting and thus the problem is obstructed.

We conclude that solving $\mathcal{F}(V(s))=0$ with $V'(0)=V_1\neq 0$ is obstructed, in our case, at the cubic term.  More generally, one may have to look at further and further terms, which is tractable since $V_n=U_n+W_n$ where $U_n\in\mathcal{K}$ and $W_n\in\mathcal{K}^{\perp}$, with $W_n$ uniquely determined by $V_1,\ldots,V_{n-1}$.  This would show that the infinitesimal deformation given by $V_1$ is obstructed.

If the problem of $\mathcal{F}(V(s))=0$ is obstructed at the $n$th term so that
$$\pi_{\mathcal{C}}\big(\mathcal{F}(V(s))\big)=\frac{s^n}{n!}\pi_{\mathcal{C}}\big(\mathcal{F}_n(V_1,\ldots,V_1)\big),$$ we see that if we write $V(s)=U(s)+W(s)$ where $U(s)\in\mathcal{K}$ and $W(s)\in\mathcal{K}^{\perp}$, then
$$\pi_{\mathcal{C}}\big(\mathcal{F}(V(s))\big)\geq \delta s^n\|U'(0)\|^n.$$
We know that $U'(0)=V_1\neq 0$, and that
$$\frac{\|U(s)-sU'(0)\|}{s}\to 0 \quad\text{as }s\to 0.$$
Hence,
$$\pi_{\mathcal{C}}\big(\mathcal{F}(V(s))\big)\geq \delta\|U(s)\|^n$$
for some $\delta>0$ and all $s$ sufficiently small. 

Therefore, $\mathcal{F}(V(s))=0$ implies that $V(s)=W(s)\in\mathcal{K}^{\perp}$.  From this point, it is straightforward to see that $\mathcal{F}(V(s))=0$ forces $V(s)=0$ for all $s$ sufficiently small.  We deduce that $0$ is the unique solution to $\mathcal{F}(V)=0$ for some small neighbourhood of $0$ in $\mathcal{X}$.

\vspace{-4pt}
\subsection{Obstructedness of the self-shrinker equation near the Clifford torus} We see that the linearisation \eqref{eq:S0.linear} of the self-shrinker operator, modulo rotations, still has a kernel, so the Clifford torus does have non-trivial infinitesimal deformations as a self-shrinker.  To show that that the Clifford torus is locally unique,  we will therefore adapt the outlined strategy and show that these infinitesimal deformations are \emph{obstructed}; that is, there are no self-shrinkers generated by them.

Before continuing on, we make some elementary observations that shall be useful later.

\begin{lem} \label{lem:bounds}  
Recall the notation of Lemma \ref{lem:shrinker}. There exists $c_{\mathcal{L}}>0$ such that for all $V\in(\Ker\mathcal{L}_0)^{\perp}\subseteq C^{2,\alpha}(NL)$ we have
\begin{equation*}
\|V\|_{C^{2,\alpha}}\leq c_{\mathcal{L}}\|\mathcal{L}_0(V)\|_{C^{0,\alpha}}.
\end{equation*}
\end{lem}

\begin{proof} This follows from follows from the ellipticity of $\mathcal{L}_0$.  
\end{proof}

We study the case of $V_A^{\perp}$ given in \eqref{eq:VAperp} in detail as the calculations for all of the other kernel elements is essentially the same.  We show that the infinitesimal deformation  $V_A^{\perp}$ is obstructed at cubic order.

\begin{prop}\label{prop:VAperp.obs}
For all $s$ sufficiently small, we have that
\begin{equation}\label{eq:S0.VAperp}
\mathcal{S}_0(sV_A^{\perp})=\frac{2s^2(1+5\cos(2\theta_1+2\theta_2))-19s^3\cos(3\theta_1+3\theta_2)}{8}(JX_1+JX_2)+\frac{s^3}{8}V_A^{\perp}+O(s^4).
\end{equation}
Thus, 
$$\pi_{\mathcal{K}}\big(\mathcal{S}_0(sV_{A}^{\perp})\big)=\frac{s^3}{8}V_A^{\perp}+O(s^4).$$
Hence, if $\mathcal{S}_{0,k}:C^{2,\alpha}(NL)\times\ldots\times C^{2,\alpha}(NL)\to C^{0,\alpha}(NL)$ is the $k$-fold symmetric map determined by
$$\mathcal{S}_{0,k}(V,\ldots,V)=\frac{\partial^k}{\partial s^k}\big(\mathcal{S}_0(sV)\big)|_{s=0}$$
then
$$\pi_{\mathcal{K}}\big(\mathcal{S}_{0,3}(V_A^{\perp},V_A^{\perp},V_A^{\perp})\big)=\frac{3}{4}V_A^{\perp}.$$
Moreover, if
$W_2\in\mathcal{K}^{\perp}$ is the unique solution to 
$$\mathcal{L}_0(W_2)=-\mathcal{S}_{0,2}(V_A^{\perp},V_A^{\perp}),$$
then 
\begin{equation}\label{eq:proj.S0.VAperp.2}
\pi_{\mathcal{K}}\big(3\mathcal{S}_{0,2}(V_A^{\perp},W_2)+\mathcal{S}_{0,3}(V_A^{\perp},V_A^{\perp},V_A^{\perp})\big)=-4V_A^{\perp}.
\end{equation}
\end{prop}

\begin{proof}
It is an elementary explicit computation to show \eqref{eq:S0.VAperp}. We give it below for completeness.

The position vector of the graph $L_{sV_A^{\perp}}$ of $sV_A^{\perp}$ over $L$ is given by
\begin{equation*}
X(s)=X+sV_A^{\perp}=\sqrt{2}\big(1+s\cos(\theta_1+\theta_2)\big)(e^{i\theta_1},e^{i\theta_2}).
\end{equation*}
We therefore have tangent vectors 
\begin{align*}
X_1(s)&=X(s)_*\left(\frac{\partial}{\partial\theta_1}\right)=\sqrt{2}\big(1+s\cos(\theta_1+\theta_2)\big)(ie^{i\theta_1},0)-\sqrt{2}s\sin(\theta_1+\theta_2)(e^{i\theta_1},e^{i\theta_2}),\\
X_2(s)&=X(s)_*\left(\frac{\partial}{\partial\theta_2}\right)=\sqrt{2}\big(1+s\cos(\theta_1+\theta_2)\big)(0,ie^{i\theta_2})-\sqrt{2}s\sin(\theta_1+\theta_2)(e^{i\theta_1},e^{i\theta_2}).
\end{align*}
Hence, the induced metric on $L_{sV_A^{\perp}}$ is
\begin{equation*}
\big(2+3s^2+4s\cos(\theta_1+\theta_2)-s^2\cos(2\theta_1+2\theta_2)\big)(\rd\theta_1^2+\rd\theta_2^2)+2s^2\big(1-\cos(2\theta_1+2\theta_2)\big)\rd\theta_1\rd\theta_2.
\end{equation*}
From this data, it is straightforward to compute $\nabla_{X_i(s)}X_j(s)$ and its projection to the normal bundle of $L_{sV_A^{\perp}}$, and hence the mean curvature $H(s)$.  Explicitly, we see that the determinant of the induced metric is
\begin{equation*}
4+16s\cos(\theta_1+\theta_2)+4s^2(5+\cos(2\theta_1+2\theta_2))+4s^3(5\cos(\theta_1+\theta_2)-\cos(3\theta_1+3\theta_2))+O(s^4),
\end{equation*}
and thus the inverse of the induced metric is
\begin{gather*}
\left(\frac{1}{2}-s\cos(\theta_1+\theta_2)+s^2\frac{5+\cos(2\theta_1+2\theta_2)}{4}-s^3\frac{\cos(\theta_1+\theta_2)+\cos(3\theta_1+3\theta_2)}{2}\right)(\rd\theta_1^2+\rd\theta_2^2)\\
+\left(s^2\frac{\cos(2\theta_1+2\theta_2)-1}{2}+s^3\big(\cos(\theta_1+\theta_2)-\cos(3\theta_1+3\theta_2)\big)\right)\rd\theta_1\rd\theta_2+O(s^4)
\end{gather*}
We also see that
\begin{align*}
\nabla_{X_1(s)}X_1(s)&=-\sqrt{2}\big((1+2s\cos(\theta_1+\theta_2))e^{i\theta_1},s\cos(\theta_1+\theta_2)e^{i\theta_2}\big)-2\sqrt{2}s\sin(\theta_1+\theta_2)(ie^{i\theta_1},0),\\
\nabla_{X_1(s)}X_2(s)&=-\sqrt{2}s\cos(\theta_1+\theta_2)(e^{i\theta_1},e^{i\theta_2})-\sqrt{2}s\sin(\theta_1+\theta_2)(ie^{i\theta_1},ie^{i\theta_2})\\
&=\nabla_{X_2(s)}X_1(s),\\
\nabla_{X_2(s)}X_2(s)&=-\sqrt{2}\big(s\cos(\theta_1+\theta_2)e^{i\theta_1},(1+2s\cos(\theta_1+\theta_2))e^{i\theta_2}\big)-2\sqrt{2}s\sin(\theta_1+\theta_2)(0,ie^{i\theta_2}).
\end{align*}
Therefore, in order to compute the normal projection, we may compute $\langle \nabla_{X_i(s)}X_j(s),X_k(s)\rangle$, for example
\begin{align*}
\langle \nabla_{X_1(s)}X_1(s),X_1(s)\rangle &
=-2s\sin(\theta_1+\theta_2)+s^2\sin(2\theta_1+2\theta_2),
\end{align*}
and thus find the mean curvature $H(s)$ from the formula
\begin{align*}
H(s)=&g^{ij}\big(\nabla_{X_i(s)}X_j(s)-g^{kl}\langle \nabla_{X_i(s)}X_j(s),X_k(s)\rangle X_l(s)\big)\\
=&-\frac{\sqrt{2}}{2}(e^{i\theta_1},e^{i\theta_2}) - \frac{\sqrt{2}s}{2}\left(\cos(\theta_1 +\theta_2) (e^{i\theta_1},e^{i\theta_2}) +2\sin(\theta_1 + \theta_2)(ie^{i\theta_1},ie^{i\theta_2})\right)
\\
&+\frac{\sqrt{2}s^2}{4}\big(5+\cos(2\theta_1 +2 \theta_2)\big)(e^{i\theta_1},e^{i\theta_2})
\\
&+\frac{\sqrt{2}s^3}{2}\bigg(-\frac{5\big(\cos(\theta_1 + \theta_2)+3\cos(3\theta_1+3\theta_2)\big)}{4} (e^{i\theta_1},e^{i\theta_2}) 
\\
&\qquad\quad+\frac{9\sin(\theta_1 +\theta_2)+\sin(3\theta_1+3\theta_2)}{2}(ie^{i\theta_1},ie^{i\theta_2})\bigg) + O(s^4).
\end{align*}
We may also compute the projection of $X(s)$ to the normal bundle of $L_{sV_A^{\perp}}$, calculating that
\begin{align*}
\frac{X(s)^{\perp}}{2}=&\frac{1}{2}\big(X(s)-g^{ij}\langle X(s),X_i(s)\rangle X_j(s)\big)\\
=&\frac{\sqrt{2}}{2}(e^{i\theta_1},e^{i\theta_2}) + \frac{\sqrt{2}s}{2}\left(\cos(\theta_1 + \theta_2) (e^{i\theta_1},e^{i\theta_2})+2\sin(\theta_1 +\theta_2)(ie^{i\theta_1},ie^{i\theta_2})\right) \\
&+\sqrt{2}s^2\big(\cos(2\theta_1 +2\theta_2)-1\big)(e^{i\theta_1},e^{i\theta_2}) \\
&+\frac{\sqrt{2}s^3}{2}\Big(\big(\cos(\theta_1 +\theta_2)-\cos(3\theta_1+3\theta_2)\big) (e^{i\theta_1},e^{i\theta_2})\\
&\qquad\quad+ 2 \big(\sin(3\theta_1+3\theta_2)-3\sin(\theta_1+\theta_2)\big) (ie^{i\theta_1},ie^{i\theta_2})\Big) + O(s^4).
\end{align*}  Hence, we have an explicit formula for $H(s)+\frac{X(s)}{2}^{\perp}$, which we can then project to the normal bundle of $L$ (after translation, which is elementary since we have written $L_{sV_A^{\perp}}$ as an exponential normal graph).  This gives us the formula \eqref{eq:S0.VAperp}.

We see that the quadratic term in \eqref{eq:S0.VAperp} is 
$$\frac{s^2}{2}\cdot\frac{(1+5\cos(2\theta_1+2\theta_2))}{2}(JX_1+JX_2)$$
and thus
$$\mathcal{S}_{0,2}(V_A^{\perp},V_A^{\perp})=\frac{(1+5\cos(2\theta_1+2\theta_2))}{2}(JX_1+JX_2).$$
Using \eqref{eq:L.Delta}, \eqref{eq:L.nDelta} and \eqref{eq:S0.linear} we know that
$$\mathcal{L}_0(f_1JX_1+f_2JX_2)=\left(\left(-\frac{1}{2}\frac{\partial^2}{\partial\theta_1^2}-\frac{1}{2}\frac{\partial^2}{\partial\theta_2^2}-1\right)f_1\right)JX_1+\left(\left(-\frac{1}{2}\frac{\partial^2}{\partial\theta_1^2}-\frac{1}{2}\frac{\partial^2}{\partial\theta_2^2}-1\right)f_2\right)JX_2.$$
Therefore, the unique element $W_2\in\mathcal{K}^{\perp}$ such that 
$$\mathcal{L}_0(W_2)=-\frac{(1+5\cos(2\theta_1+2\theta_2))}{2}(JX_1+JX_2)=-\mathcal{S}_{0,2}(V_A^{\perp},V_A^{\perp})$$
is given by
$$W_2=\frac{3-5\cos(2\theta_1+2\theta_2)}{6}(JX_1+JX_2).$$

Our task now is to compute
$$\mathcal{S}_{0,2}(V_A^{\perp},W_2).$$
To do this, it is more streamlined to consider the following general situation, where we take normal vector fields $V$ as follows:
\begin{equation}\label{eq:V.f}
V=\sqrt{2} f(\theta_1,\theta_2)(e^{i\theta_1},e^{i\theta_2}).
\end{equation}
Note that both $V_A^{\perp}$ and $W_2$ are of this form, with 
\begin{equation}\label{eq:f1f2.VAW2}
f_1=\cos(\theta_1+\theta_2)\quad\text{and}\quad f_2=\frac{5\cos(2\theta_1+2\theta_2)-3}{6}
\end{equation}
for $V_A^{\perp}$ and $W_2$ respectively.

We can then consider the graph of $L_{sV}$ over $L$, just as when we took $V=V_A^\perp$ above, which now has position vector
\begin{align*}
X(s)&=\sqrt{2}\big(1+sf\big)(e^{i\theta_1},e^{i\theta_2})
\end{align*}
We may then compute the tangent vectors to $L_{sV}$:
\begin{align*}
X_1(s)&=\sqrt{2}(1+sf)(ie^{i\theta_1},0)+\sqrt{2} s\textstyle\frac{\partial f}{\partial\theta_1}(e^{i\theta_1},e^{i\theta_2})\\
X_2(s)&=\sqrt{2}(1+sf)(0,ie^{i\theta_2})+\sqrt{2}s\textstyle\frac{\partial f}{\partial\theta_2}(e^{i\theta_1},e^{i\theta_2}).
\end{align*}
We therefore find the induced metric on $L_{sV}$:
\begin{align*}
g_{11}&=\langle X_1(s),X_1(s)\rangle =2(1+sf)^2+4s^2\big(\textstyle\frac{\partial f}{\partial\theta_1}\big)^2,\\
g_{12}&=\langle X_1(s),X_2(s)\rangle=4s^2\textstyle\frac{\partial f}{\partial\theta_1}\frac{\partial f}{\partial \theta_2},\\
g_{22}&=\langle X_2(s),X_2(s)\rangle =2(1+sf)^2+4s^2\big(\textstyle\frac{\partial f}{\partial\theta_2}\big)^2.
\end{align*}
Thus, the entries $(g^{ij})$ of the inverse of the induced metric are
\begin{align*}
g^{11}&=\frac{1}{2}-sf+\frac{s^2}{2}\big(3f^2-2\textstyle(\frac{\partial f}{\partial \theta_1})^2\big)+O(s^3)
,\\
g^{12}&=-s^2\textstyle\frac{\partial f}{\partial \theta_1}\frac{\partial f}{\partial \theta_2}+O(s^3) 
, \\
g^{22}&=\frac{1}{2}-sf+\frac{s^2}{2}\big(3f^2-2\textstyle(\frac{\partial f}{\partial \theta_2})^2\big)+O(s^3) 
.
\end{align*}
We now compute $\nabla_{X_i(s)}X_j(s)$ as follows:
\begin{align*}
\nabla_{X_1(s)}X_1(s)&=-\sqrt{2}\big(\big(1+sf-s\textstyle\frac{\partial^2f}{\partial\theta_1^2}\big)e^{i\theta_1},-s\textstyle\frac{\partial^2f}{\partial\theta_1^2} e^{i\theta_2}\big)+2\sqrt{2}s\textstyle\frac{\partial f}{\partial\theta_1}(ie^{i\theta_1},0),\\
\nabla_{X_1(s)}X_2(s)&=\sqrt{2}s\textstyle\frac{\partial^2f}{\partial\theta_1\partial\theta_2}(e^{i\theta_1},e^{i\theta_2})+\sqrt{2}s\big(\frac{\partial f}{\partial\theta_2}ie^{i\theta_1},\frac{\partial f}{\partial\theta_1}e^{i\theta_2}\big)\\
&=\nabla_{X_2(s)}X_1(s),\\
\nabla_{X_2(s)}X_2(s)&=-\sqrt{2}\big(-s\textstyle\frac{\partial^2f}{\partial\theta_2^2} e^{i\theta_1},\big(1+sf-s\textstyle\frac{\partial^2f}{\partial\theta_2^2}\big)e^{i\theta_1}\big)+2\sqrt{2}s\textstyle\frac{\partial f}{\partial\theta_2}(0,ie^{i\theta_2}).
\end{align*}
Given this, it is straightforward to calculate the 
inner products $\langle\nabla_{X_i(s)}X_j(s),X_k(s)\rangle$, such as:
\begin{align*}
\langle\nabla_{X_1(s)}X_1(s),X_1(s)\rangle&=2s\textstyle\frac{\partial f}{\partial\theta_1}+2s^2\textstyle\frac{\partial f}{\partial\theta_1}\big(f+2\frac{\partial^2f}{\partial\theta_1^2}\big)+O(s^4),\\
\langle\nabla_{X_1(s)}X_1(s),X_2(s)\rangle&=-2s\textstyle\frac{\partial f}{\partial\theta_2}-2s^2\textstyle\frac{\partial f}{\partial\theta_2}\big(f-2\frac{\partial^2f}{\partial\theta_1^2}\big)+O(s^4).
\end{align*}
From our computations so far, it is now straightforward to find the mean curvature $H(s)$ of $L_{sV}$ and project it to the normal bundle of $L$.  We see that
\begin{align*}
H(s)&=g^{ij}\big(\nabla_{X_i(s)}X_j(s)-g^{kl}\langle\nabla_{X_i(s)}X_j(s),X_k(s)\rangle X_l(s)\big)\\
&=\sqrt{2}\left(-\frac{1}{2}+\frac{s}{2}\left(f+\frac{\partial^2 f}{\partial\theta_1^2}+\frac{\partial^2f}{\partial\theta_2^2}\right)-\frac{s^2}{2}f\left(f+2\frac{\partial^2 f}{\partial\theta_1^2}+2\frac{\partial^2f}{\partial\theta_2^2}\right)\right)(e^{i\theta_1},e^{i\theta_2})\\
&+\sqrt{2}s^2\left(\left(\frac{\partial f}{\partial\theta_1}\right)^2e^{i\theta_1},\left(\frac{\partial f}{\partial\theta_2}\right)^2e^{i\theta_2}\right)+T_1(s),
\end{align*}
where $T_1(s)$ is tangent to $L$ (once it is pulled back to $L$) plus terms of order $O(s^3)$.

We also see that
\begin{align*}
\langle X(s),X_1(s)\rangle &= 4s(1+sf)\textstyle\frac{\partial f}{\partial\theta_1},\\
\langle X(s),X_2(s)\rangle &= 4s(1+sf)\textstyle\frac{\partial f}{\partial\theta_2}.
\end{align*}
We deduce that
\begin{align*}
\frac{X(s)^{\perp}}{2}&=\frac{1}{2}(X(s)-g^{ij}\langle X(s),X_i(s)\rangle X_j(s))\\
&=\sqrt{2}\left(\frac{1}{2}+\frac{s}{2}f-s^2\left(\left(\frac{\partial f}{\partial\theta_1}\right)^2+\left(\frac{\partial f}{\partial\theta_2}\right)^2\right)\right)(e^{i\theta_1},e^{i\theta_2})+T_2(s),
\end{align*}
where, again, $T_2(s,t)$ is tangent to $L$ plus terms of order $O(s^3)$.

Overall, we see that
\begin{align}
\mathcal{S}_{0,2}(V,V)&=-\left(f\left(f+2\frac{\partial^2f}{\partial\theta_1^2}+2\frac{\partial^2f}{\partial\theta_2^2}\right)\right)(JX_1+JX_2)\nonumber\\
&-2\left(\frac{\partial f}{\partial\theta_2}\right)^2JX_1-2\left(\frac{\partial f}{\partial\theta_1}\right)^2JX_2
.\label{eq:S02}
\end{align}
We can then find $\mathcal{S}_{0,2}(V_A^{\perp},W_2)$ by polarizing \eqref{eq:S02} and using \eqref{eq:f1f2.VAW2}, i.e.
\begin{align*}
\mathcal{S}_{0,2}(V_A^{\perp},W_2)&=\frac{1}{2}\left(\mathcal{S}_{0,2}(V_A^{\perp}+W_2,V_A^{\perp}+W_2)-\mathcal{S}_{0,2}(V_A^{\perp},V_A^{\perp})-\mathcal{S}_{0,2}(W_2,W_2)\right)\\
&=\left(-f_1f_2-f_1\left(\frac{\partial^2 f_2}{\partial\theta_1^2}+\frac{\partial^2f_2}{\partial\theta_2^2}\right)-f_2\left(\frac{\partial^2 f_1}{\partial\theta_1^2}+\frac{\partial^2f_1}{\partial\theta_2^2}\right)\right)(JX_1+JX_2)\\
&\qquad-2\frac{\partial f_1}{\partial\theta_2}\frac{\partial f_2}{\partial\theta_2}JX_1
-2\frac{\partial f_1}{\partial\theta_1}\frac{\partial f_2}{\partial\theta_1}JX_2\\
&=\frac{19\cos(\theta_1+\theta_2)+65\cos(3\theta_1+3\theta_2)}{12}(JX_1+JX_2)
\end{align*}
Hence,
$$\pi_{\mathcal{K}}\big(\mathcal{S}_{0,2}(V_A^{\perp},W_2)\big)=-\frac{19}{12}V_A^{\perp}.$$
Equation \eqref{eq:proj.S0.VAperp.2} then follows.
\end{proof}

\begin{rem}
In the proof of Proposition \ref{prop:VAperp.obs}, we see that it is crucial that we know $W_2$ precisely in terms of $V_A^{\perp}$  to obtain that the right-hand side of \eqref{eq:proj.S0.VAperp.2} is non-zero.  In particular, by taking a different multiple of $W_2$ in \eqref{eq:proj.S0.VAperp.2} we can force the right-hand side to be zero.
\end{rem}


\begin{prop}\label{prop:obs} 
Recall the notation $\mathcal{S}_{0,k}$ from Proposition \ref{prop:VAperp.obs}.   
Let $U\in\mathcal{K}$ with $\|U\|_{C^{2,\alpha}}=1$ and let $W_2\in\mathcal{K}^{\perp}$ be the unique solution to
$$
\mathcal{L}_0(W_2)=-\mathcal{S}_{0,2}(U,U).
$$ 
There exists a constant $\delta>0$, independent of $U$, such that
\begin{equation}\label{eq:S0.U}
\|\mathcal{S}_{0,2}(U,U)\|_{C^{0,\alpha}}\geq \delta, \quad\pi_{\mathcal{K}}(\mathcal{S}_{0,2}(U,U))=0,
\quad
\|\pi_{\mathcal{K}}(\mathcal{S}_{0,3}(U,U,U))\|_{C^{0,\alpha}}\geq\delta,
\end{equation}
and
\begin{equation}\label{eq:U.obs}
\|\pi_{\mathcal{K}}\big(3\mathcal{S}_{0,2}(U,W_2)+\mathcal{S}_{0,3}(U,U,U)\big)\|_{C^{0,\alpha}}\geq\delta.
\end{equation}
\end{prop}

\begin{proof}
Let $a,b,c,d\in\bR$ so that
\begin{equation*}
U=aV_A^{\perp}+bV_B^{\perp}+cV_C^{\perp}+dV_D^{\perp}.
\end{equation*}
We explicitly compute $\mathcal{S}_0(sU)$, as we did for $V_A^{\perp}$ in the proof of Proposition \ref{prop:VAperp.obs} above (i.e.~the case $(a,b,c,d)=(1,0,0,0)$),  by setting
\begin{equation*}\label{eq:U.f}
f(\theta_1,\theta_2)=a\cos(\theta_1+\theta_2)+b\sin(\theta_1+\theta_2)+c\cos(\theta_1-\theta_2)-d\sin(\theta_1-\theta_2)
\end{equation*}
in \eqref{eq:V.f}.
In particular, one finds that
\begin{align*}
\pi_{\mathcal{K}}\circ\mathcal{S}_{0}(sU)&=\frac{s^3}{8}(a^2+b^2+18c^2+18d^2)(aV_A^{\perp}+b V_B^{\perp})\\
&\quad+\frac{s^3}{8}(18a^2+18b^2+c^2+d^2)(cV_C^{\perp}+dV_D^{\perp})+O(s^4).
\end{align*}
Equation \eqref{eq:S0.U} quickly follows from this formula and the explicit computation of $\mathcal{S}_{0,2}(U,U)$.
%
%

Equation \eqref{eq:U.obs} now follows from calculations just as in the proof of Proposition \ref{prop:VAperp.obs}.
\end{proof}

\subsection{Main result}

We now have all of the ingredients necessary to prove our local uniqueness result for the Clifford torus.  We consider here self-shrinkers with arbitrary space-time centres $(x_0,t_0) \in \bC^2\times\bR^+$, i.e.~satisfying \eqref{eq:shrinker2}.

\begin{thm}\label{thm:local.uniq}
Any 2-dimensional compact embedded self-shrinker in $\bC^2$ which is sufficiently $C^{2,\alpha}$-close to the Clifford torus $L$ is, up to some translation, dilation and rotation, equal to $L$.   
\end{thm}

\begin{proof}
Assume first that we have a sequence of self-shrinkers $L^j$ with centers $(x_0^j, t_0^j)$ converging in $C^{2,\alpha}$ to $L$.  Then we have  $(x_0^j, t_0^j) \rightarrow (0,1)$ and the shrinkers
$$ \tilde{L}^j = (t^j_0)^{-1/2}\left( L^j -x^j_0\right)
$$
also converge in $C^{2,\alpha}$ to $L$. It is thus sufficient to replace the sequence $L^j$ by the sequence of self-shrinkers $\tilde{L}^j$, which have centres at $(0,1)$.

To show that for large enough $j$, we have that up to a rotation $L^j = L$, we know by Lemma \ref{lem:shrinker} that we must show that the only solution to $\mathcal{S}_{0}(V)=0$ for $\|V\|_{C^{2,\alpha}}$ sufficiently small is $V=0$. We can therefore assume, for a contradiction, that we have sequence $V_k \rightarrow 0$ in $C^{2,\alpha}$ satisfying $\mathcal{S}_{0}(V_k)=0$. We now adapt the strategy laid out in Section \ref{sec:def and ob} to obtain our desired contradiction.

Since $\mathcal{S}_0$ is analytic, there exist $i$-linear forms $\mathcal{S}_{0,i}$, $i=2,3$, such that
$$\mathcal{S}_{0}(V) = \mathcal{L}_{0}(V) + \frac{\mathcal{S}_{0,2}}{2} (V,V) + \frac{\mathcal{S}_{0,3}}{6} (V,V,V) + O(\|V\|_{C^{2,\alpha}}^4)\, .$$

We can assume that   $V_k=s_kU_k+t_kW_k$ where $s_k\neq 0$, $U_k\in\mathcal{K}$ with $\|U_k\|_{C^{2,\alpha}}=1$, $W_k\in\mathcal{K}^{\perp}$ with $\|W_k\|_{C^{2,\alpha}}=1$, $(s_k,t_k)\to 0$ and $s_k>0$ and $t_k\geq 0$ (since we can just swap the sign of $U_k,W_k$ otherwise).
We have
\begin{align*}
0=\mathcal{S}(V_k) &=s_k\mathcal{L}_{0}(U_k)+t_k\mathcal{L}_0(W_k)
+\frac{s_k^2}{2}\mathcal{S}_{0,2}(U_k,U_k)+s_kt_k\mathcal{S}_{0,2}(U_k,W_k)\\
&\quad +\frac{t_k^2}{2}\mathcal{S}_{0,2}(W_k,W_k) +O(s_k^3,s_kt_k^2,s_k^2t_k,t_k^3).
\end{align*}

As $\mathcal{L}_0(U_k)=0$, and we know from \eqref{eq:S0.U} that $\mathcal{S}_{0,2}(U,U)\neq 0$ for all $U\in\mathcal{K}$ with $\|U\|_{C^{2,\alpha}}=1$, we must have $t_k> 0$ for all $k$. Since $U_k\in\mathcal{K}$ for all $k$, we see that
$$0=\mathcal{L}_0(W_k)+\frac{s_k^2}{2t_k}\mathcal{S}_{0,2}(U_k,U_k)+O(t_k, s_k ,s_k \frac{s_k^2}{t_k}).$$

From Lemma \ref{lem:bounds} we have
$$\|\mathcal{L}_0(W_k)\|_{C^{0,\alpha}}\geq c,$$
 and by \eqref{eq:S0.U}
$$\|\mathcal{S}_{0,2}(U_k,U_k)\|_{C^{0,\alpha}}\geq c'>0$$
for all $k$.  Since  $(s_k,t_k)\to 0$, we can therefore assume that
$$\frac{t_k}{s_k^2}\to \lambda> 0$$
 as $k\to 0$.  
Therefore, 
$$t_k=\lambda s_k^2+u_k$$
where 
$$\frac{u_k}{s_k^2}\to 0.$$

We deduce that
$$\mathcal{L}_0(W_k)+\frac{1}{2\lambda}\mathcal{S}_{0,2}(U_k,U_k)\to 0,$$
in $C^{0,\alpha}$. Since  $\mathcal{S}_{0,2}(U_k,U_k)\in\Imm\mathcal{L}_0$ by \eqref{eq:S0.U}, this determines $W_k$ uniquely up to an error that tends to $0$ as $k\to 0$: i.e.~if we set 
$$W_k^0=-\frac{1}{2\lambda}\mathcal{L}_0^{-1}\circ\mathcal{S}_{0,2}(U_k,U_k)\, ,$$
then
$$W_k=W_k^0+E_k$$
where
$$\|E_k\|_{C^{2,\beta}}\to 0.$$
Furthermore, note that $E_k \in \mathcal{K}^\perp$. Therefore,
\begin{align*}
0&=\mathcal{L}_0(E_k)+\left(\frac{s_k^2}{2\lambda s_k^2+2u_k}-\frac{1}{2\lambda}\right)\mathcal{S}_{0,2}(U_k,U_k)+O(s_k)\\
&=\mathcal{L}_0(E_k)-\frac{u_k}{2\lambda^2 s_k^2}\mathcal{S}_{0,2}(U_k,U_k)+O\left(s_k,\frac{u_k^2}{s_k^4}\right).
\end{align*}
We deduce that 
\begin{align*}
 \|\mathcal{L}_0(E_k)\|_{C^{0,\alpha}},\|E_k\|_{C^{2,\alpha}}\leq C\max\left\{\frac{|u_k|}{s_k^2},s_k\right\}.
\end{align*}
We then see that
\begin{align*}
0&=s_k^2\left(\lambda\mathcal{L}_0(W_k)+\frac{1}{2}\mathcal{S}_{0,2}(U_k,U_k)\right)\\
&\quad +u_k\mathcal{L}_0(W_k)+\lambda s_k^3\mathcal{S}_{0,2}(U_k,W_k)+\frac{s_k^3}{6}\mathcal{S}_{0,3}(U_k,U_k,U_k)+o(s_k^3)\\
&=\lambda s_k^2\mathcal{L}_0(E_k)+u_k\mathcal{L}_0(W_k)\\
&\quad +\frac{s_k^3}{6}(6\lambda \mathcal{S}_{0,2}(U_k,W_k)+\mathcal{S}_{0,3}(U_k,U_k,U_k)) +o(s_k^3).
\end{align*} 
Taking the projection onto $\mathcal{K}$ we see that
\begin{align*}
0=\pi_{\mathcal{K}}(6\lambda \mathcal{S}_{0,2}(U_k,W_k)+\mathcal{S}_{0,3}(U_k,U_k,U_k))+o(1).
\end{align*}
We also know that $W_k=W_k^0+E_k$ where $\|E_k\|_{C^{2,\alpha}}\to 0$, so we deduce that
$$\pi_{\mathcal{K}}(3\mathcal{S}_{0,2}(U_k,2\lambda W_k^0)+\mathcal{S}_{0,3}(U_k,U_k,U_k))\to 0.$$
Recall that $W_k^0\in\mathcal{K}^{\perp}$ is uniquely determined by $U_k$ via $$\mathcal{L}_0(2\lambda W_k^0)=-\mathcal{S}_{0,2}(U_k,U_k),$$ and we have from \eqref{eq:U.obs} in Proposition \ref{prop:obs} that
$$\|\pi_{\mathcal{K}}(3\mathcal{S}_{0,2}(U_k,2\lambda W_k^0)+\mathcal{S}_{0,3}(U_k,U_k,U_k))\|_{C^{0,\alpha}}\geq c>0$$
since $\|U_k\|_{C^{2,\alpha}}=1$, which yields a contradiction.  
 
This means that we must have $s_k=0$ for all $k$, i.e.~that $V_k\in\mathcal{K}^{\perp}$ for all $k$. Up to a subsequence we can assume that $V_k \rightarrow V \in\mathcal{K}^{\perp}$ in $C^{2,\beta}$. This yields a contradiction since necessarily in the limit $V$ has to satisfy $\mathcal{L}_0(V) = 0$. 
\end{proof}


\providecommand{\bysame}{\leavevmode\hbox to3em{\hrulefill}\thinspace}
\providecommand{\MR}{\relax\ifhmode\unskip\space\fi MR }
\providecommand{\MRhref}[2]{%
  \href{http://www.ams.org/mathscinet-getitem?mr=#1}{#2}
}
\providecommand{\href}[2]{#2}

\end{document}